\documentclass[11pt,a4paper]{article}
\usepackage[utf8]{inputenc}

\usepackage{mathptmx}

\usepackage[T1]{fontenc}

\usepackage{amsmath}
\usepackage{amsfonts}
\usepackage{amssymb}
\usepackage{amsthm}

\usepackage[left=2cm,right=2cm,top=2cm,bottom=2cm]{geometry}

\usepackage{hyperref}
\usepackage{mathtools}
\mathtoolsset{showonlyrefs}

\usepackage{caption}
\usepackage[labelformat=simple]{subcaption}

\usepackage{enumerate}

\usepackage{cite}
\makeatletter
\newcommand{\citecomment}[2][]{\citen{#2}#1\citevar}
\newcommand{\citeone}[1]{\citecomment{#1}}
\newcommand{\citetwo}[2][]{\citecomment[,~#1]{#2}}
\newcommand{\citevar}{\@ifnextchar\bgroup{;~\citeone}{\@ifnextchar[{;~\citetwo}{]}}}
\newcommand{\citefirst}{\@ifnextchar\bgroup{\citeone}{\@ifnextchar[{\citetwo}{]}}}
\newcommand{\cites}{[\citefirst}
\makeatother

\theoremstyle{plain}
\newtheorem{thm}{Theorem}[section]
\newtheorem{lemma}[thm]{Lemma}
\newtheorem{prop}[thm]{Proposition}
\newtheorem{cor}[thm]{Corollary}

\theoremstyle{remark}
\newtheorem{remark}[thm]{Remark}
\newtheorem{hyp}[thm]{Assumption}

\theoremstyle{definition}

\newtheorem{definition}[thm]{Definition}

\usepackage{bbm}

\newcommand{\N}{\mathbb{N}}
\newcommand{\Z}{\mathbb{Z}}
\newcommand{\PP}{\mathbb{P}}
\newcommand{\EE}{\mathbb{E}}
\newcommand{\RR}{\mathbb{R}}
\newcommand{\R}{\mathbb{R}}

\newcommand{\TT}{\mathbb{T}}
\newcommand{\M}{\mathbb{M}}

\newcommand{\cF}{\mc{F}}
\newcommand{\cP}{\mc{P}}

\newcommand{\one}{\mathbbm{1}}

\newcommand{\e}{\varepsilon}

\newcommand{\mc}{\mathcal}
\newcommand{\mb}{\mathbb}

\newcommand{\be}{\begin{equation}}
\newcommand{\ee}{\end{equation}}

\newcommand*\dd{\mathop{}\!\mathrm{d}}

\renewcommand{\div}{\operatorname{div}}
\newcommand{\loc}{\text{loc}}

\DeclareMathOperator*{\Var}{Var}

\numberwithin{thm}{section}
\numberwithin{equation}{section}

\newcommand{\ds}{\displaystyle}

\title{A Probabilistic Mean-Field Limit for the Vlasov-Poisson System for Ions}
\author{Megan Griffin-Pickering \thanks{University of Z\"urich, Institute of Mathematics, Winterthurerstrasse 190, 8057 Z\"urich, Switzerland. \mbox{Email: \textsf{megan.griffin-pickering@math.uzh.ch}}}}
\begin{document}

\maketitle

\begin{abstract}
The Vlasov-Poisson system for ions is a kinetic equation for dilute, unmagnetised plasma. It describes the evolution of the ions in a plasma under the assumption that the electrons are thermalized. Consequently, the Poisson coupling for the electrostatic potential contains an additional exponential nonlinearity not present in the electron Vlasov-Poisson system.

The system can be formally derived through a mean-field limit from a microscopic system of ions interacting with a thermalized electron distribution. However, it is an open problem to justify this limit rigorously for ions modelled as point charges. 
Existing results on the derivation of the three-dimensional ionic Vlasov-Poisson system, obtained by the author and Iacobelli [J. Math. Pures Appl. 135 (2020), pp. 199–255], require a truncation of the singularity in the Coulomb interaction at spatial scales of order $N^{- \beta}$ with $\beta < 1/15$, which is more restrictive than the available results for the electron Vlasov-Poisson system. 

In this article, we prove that the Vlasov-Poisson system for ions can be derived from a microscopic system of ions and thermalized electrons with interaction truncated at scale $N^{- \beta}$ with $\beta < 1/3$.
We develop a generalisation of the probabilistic approach to mean-field limits developed in the works of Boers and Pickl  [J. Stat. Phys. 164(1) (2016), pp. 1--16] and Lazarovici and Pickl [Arch. Ration. Mech. Anal. 225(3) (2017), pp. 1201--1231] that is applicable to interaction forces defined through a nonlinear coupling with the particle density.
The proof is based on a quantitative uniform law of large numbers for convolutions between empirical measures of independent, identically distributed random variables and locally Lipschitz functions.
\end{abstract}

\section{Introduction}

The Vlasov-Poisson system for ions (also known as the Vlasov-Poisson system \emph{with massless electrons}) is a kinetic model for dilute plasma, in which positively charged ions interact electrostatically both with each other and with a population of (`massless') electrons assumed to be thermalized. The system with periodic boundary conditions reads as follows:
\be \label{eq:vpme}
\begin{cases}
\partial_t f + v \cdot \nabla_x f - \nabla_x \Phi [ \rho _f ] \cdot \nabla_v f = 0, \\
- \Delta_x \Phi[\rho_f] = \rho_f - e^{\Phi[\rho_f]}, \\
\ds \rho_f(t,x) : = \int_{\R^d} f(t,x,v) \dd v \\
f \vert_{t=0} = f_0 \geq 0, \; \ds \int_{\TT^d \times \R^d} f_0 \dd x \dd v = 1 .
\end{cases} \qquad (x,v) \in \TT^d \times \R^d
\ee
The unknown $f = f(t,x,v)$ is interpreted as the phase space distribution of ions at time $t$. The system \eqref{eq:vpme} may in principle be posed in any dimension $d \geq 1$; however, the physical case $d=3$ is of particular interest.

The key difference from the well-known electron model (\cite{Vlasov}, see also \cite{Jeans} for the gravitational setting) is the exponential nonlinearity in the Poisson-Boltzmann equation for the potential $\Phi = \Phi[\rho_f]$.
This term represents the distribution of electrons in thermal equilibrium, following a \emph{Maxwell-Boltzmann law} \cites[Section 1.6]{Bellan}[Section 2.4]{Lieberman-Lichtenberg}, and creates several additional mathematical difficulties: of particular relevance to the present article is the lack of an explicit representation of the interaction force as a convolution between the ion density $\rho_f$ and a given kernel, which precludes the use of techniques that rely on such a representation.

It is a fundamental problem to understand how kinetic partial differential equations such as \eqref{eq:vpme} arise from the underlying particle systems they are intended to describe. In this article, we consider the mean-field limit for the system \eqref{eq:vpme}. In our main result, we derive \eqref{eq:vpme} rigorously from an $N$-particle system of ions with cut-off Coulomb interaction and thermalized electrons. We can allow the cut-off radius to vanish at rate $N^{-1/d + \e}$ for any $\e > 0$.

\begin{paragraph}{The Massless Electrons Model.}
The assumption that, from the ions' point of view, electrons within a plasma are distributed according to the Maxwell-Boltzmann law is justified physically
by the observation that the mass ratio between a single electron ($m_e$) and a single ion ($m_i$) is very small: for example, for hydrogen $m_e / m_i$ is roughly $5 \times 10^{-4}$.
Consequently, the average speed of the electrons is typically much faster than that of the ions, leading to a separation of timescales. Moreover, the electron collision frequency can be significant on timescales where ion-ion collisions are negligible \cite[Section 1.9]{Bellan}.
Therefore, on such timescales,
convergence of the electron distribution towards thermal equilibrium is expected, with the rate of this convergence being faster for smaller mass ratio.

In the \emph{massless electron} regime ($m_e / m_i \to 0$), the electrons are expected to assume the thermal equilibrium distribution essentially instantaneously: this equilibrium distribution is of the form
\be
f_{\text{electron}}(x,v) = \mc{Z}^{-1} \exp \left ( - \frac{1}{k_B T_e} [ q_e \tilde \Phi(x) + \frac{1}{2} |v|^2 ] \right ), \quad - \varepsilon_0 \Delta_x \tilde \Phi = q_i \rho_{\text{ion}} + q_e \rho_{f_{\text{electron}}}
\ee
where $\tilde \Phi$ represents the electrostatic potential in physical units, $\rho_{\text{ion}}$ is the spatial density of ions, $q_e$ and $q_i$ are the respective charges of an electron and an ion, $T_e$ is the electron temperature, $k_B$ is the Boltzmann constant, $\varepsilon_0$ is the vacuum permittivity, and the partition function $\mc{Z}$ is a normalising constant chosen such that $f_{\text{electron}}$ has total integral one (see e.g. \cite{GPIProcWP} for an informal derivation of this relation in the massless electron limit from a coupled electron-ion system). Since $q_e < 0$, rescaling of the physical constants results in the system \eqref{eq:vpme}, with the term $e^\Phi$ in the Poisson-Boltzmann equation representing the spatial density of electrons; note that, on the torus, the normalisation $\int_{\TT^d} e^\Phi \dd x = 1$ is enforced by the fact that the Poisson equation $\Delta_x \Phi = g$ can only have a solution if the source term $g$ has integral zero (equivalently, by addition of a constant to $\Phi$). 

The model \eqref{eq:vpme} appears in the physics literature \cite{GurevichPitaevsky75}, including in the modelling of phenomena such as ion-acoustic shocks \cite{BPLT1991, Mason71, SCM} and the expansion of plasma into vacuum \cite{Medvedev2011}.

The massless electron limit has been studied from a mathematical point of view, with several works deriving models of a similar type to \eqref{eq:vpme} from various coupled systems describing the interaction of ions and electrons within a plasma.
In \cite{BGNS} a system similar to \eqref{eq:vpme} (with time-dependent electron temperature) was derived from a Vlasov-Poisson-Boltzmann system incorporating ion-ion and electron-electron collisions, under a priori bounds on the regularity of solutions to the collisional equations.
Long-time convergence to equilibrium of the Vlasov-Poisson-Fokker-Planck system was proved in \cite{BouchutDolbeault}, and the massless electron limit from a coupled Vlasov-Poisson system with external magnetic field and Fokker-Planck electron-electron collision term was studied in \cite{Herda}.
Recently \cite{FlynnGuo, Flynn}, the Vlasov-Poisson-Landau system for ions was derived in the massless electron limit from a coupled ion-electron system with Landau collision terms, including \emph{cross} collision terms between ions and electrons.

The well-posedness theory of \eqref{eq:vpme} was largely open until relatively recently. Global existence of $L^p$ weak solutions was shown for the model posed on $\R^3$ by Bouchut \cite{Bouchut}, however these solutions were not known to conserve either energy or regularity of the initial data, and no uniqueness result was available. 
Global well-posedness of classical $C^1$ solutions was shown for the one-dimensional torus $\TT^1$ by Han-Kwan and Iacobelli \cite{HKI1}.
In three dimensions, global well-posedness for classical solutions with large data was proved by the author and Iacobelli: see \cite{GPIWP} for the cases $\TT^2, \TT^3$ and \cite{GPIWPR3} for $\R^3$.
The case of bounded convex 3D domains was considered by Cesbron and Iacobelli \cite{CesbronIacobelli}.

See \cite{GPIProcWP} for a more in-depth discussion of the modelling and well-posedness theory of \eqref{eq:vpme}, including a more detailed bibliography and context of analogous results for the electron Vlasov-Poisson system.

\end{paragraph}

\begin{paragraph}{The Mean-Field Limit for Vlasov-Poisson and Related Systems.}
The subject of mean-field limits for deriving PDEs describing large many-particle systems boasts a vast and well-developed literature.
For an overview of the various possible approaches and applications, we refer the reader to the excellent surveys \cite{Golse, JabinReview, JabinWangReview, ChaintronDiez1, ChaintronDiez2}.
Here we will focus principally on particle systems obeying deterministic second-order dynamics with singular potentials.

An archetypal model takes the form
\be \label{eq:GeneralParticles}
\begin{cases}
\frac{\dd}{\dd t} X_i = V_i \\
\frac{\dd}{\dd t} V_i = - \frac{1}{N} \sum_{j \neq i} \nabla \Psi (X_i - X_j)
\end{cases}
\quad 
i \in \{1, \ldots, N \} ,
\ee
for some interaction potential $\Psi$. If $\Psi$ is taken to be the Coulomb potential, then \eqref{eq:GeneralParticles} represents a system of $N$ point charges experiencing electrostatic interactions with each other and a fixed uniform background of ions. The formal limit as $N\to \infty$ is the
electron Vlasov-Poisson system.
However, the rigorous justification of this limit remains an open problem in full generality.

The mean-field limit was proved for 
Lipschitz interaction forces ($\nabla \Psi \in W^{1, \infty}$) in the works \cite{NeunzertWick, BraunHepp, Dobrushin}.
The Vlasov-Poisson case is excluded from these results since the Coulomb force kernel $K$ has a point singularity for vanishing inter-particle separation ($|K(x)| \sim |x|^{-2}$ as $|x| \to 0$ in dimension $d=3$).
The mean-field limit for the one-dimensional electron Vlasov-Poisson system, where the Coulomb singularity is a jump discontinuity, was proved by Hauray \cite{Hauray14}.
For forces with `weak' singularities of type $\nabla \Psi(x) \sim |x|^{-\alpha}$ with $0 < \alpha < 1$ the mean limit was obtained by Hauray and Jabin \cite{HaurayJabin07, HaurayJabin15}.
See also Jabin and Wang \cite{JabinWang} for the case $\nabla \Psi \in L^\infty$. 
More recently, Bresch, Duerinckx and Jabin \cite{BreschDuerinckxJabin} proved the mean-field limit for forces in $L^2 + L^\infty$ (here corresponding to $\alpha < \frac{d}{2}$) through a novel duality method.

For stronger singularities, many authors have considered interaction kernels with an \emph{$N$-dependent truncation} that is removed simultaneously with the mean-field limit. Hauray and Jabin \cite{HaurayJabin15} derived the Vlasov equation for the cases $\alpha \in (1, d-1)$ from regularised $N$-particle systems with interaction kernel truncated at scale $|x| \lesssim N^{-\beta}$ with $0 < \beta < \frac{1}{2d} \min \{ \frac{d-2}{\alpha - 1}, \frac{2 d-1}{\alpha} \}$.
In the Coulomb case $\alpha = d-1$, Kiessling \cite{Kiessling} obtained the limit (without rate) assuming an a priori uniform-in-$N$ moment bound on the microscopic forces, whose validity for typical initial data is however at present open. 
Lazarovici \cite{Lazarovici} derived the electron Vlasov-Poisson system from an `extended charges' model with interaction truncated at order $|x| \sim N^{- \frac{1}{d(d+2)} + \e}$ for any $\e > 0$. 

The results discussed so far use methods that are predominantly `deterministic' in their treatment of the particle dynamics, with randomness appearing mainly in showing that the initial data for which the mean-field limit holds are typical. 
Boers and Pickl \cite{BoersPickl} developed a `probabilistic' approach to the mean-field limit, in which stochasticity of the initial data is used along the flow to control the dynamical evolution.
Using this approach they could obtain the limit for singularities of order $\alpha < 2$ (in dimension $d=3$) with a cut-off radius of order $N^{-1/3}$---notably, this corresponds to the scale of the typical separation of particles. 
Lazarovici and Pickl \cite{LazaroviciPickl} were then able to extend the method to cover the missing extremal case $\alpha = 2$, proving a mean-field derivation for the 3D electron Vlasov-Poisson system with cut-off at order $N^{-1/3 + \e}$. 
The cut-off was subsequently improved to $N^{-7/18 + \e}$ in the thesis of Grass \cite{GrassThesis} and to $N^{-5/12 + \e}$ by Feistl-Held and Pickl \cite{FHPickl253D}. By similar techniques, in two dimensions, Feistl-Held and Pickl  \cite{FHPickl252D} have shown the mean-field limit for the 2D electron Vlasov-Poisson system for arbitrarily small cut-off of scale $N^{-\beta}$ for any $0 < \beta \leq 2$.

We conclude the discussion of the mean-field limit for electron Vlasov-Poisson by mentioning some results not requiring truncation.
In the monokinetic case, Serfaty and Duerinckx \cite{SerfatyDuerinckx} proved the derivation of the pressureless Euler-Poisson system in the mean-field limit from second-order Coulomb dynamics.
For the Vlasov-Poisson-Fokker-Planck system, where the microscopic system includes independent idiosyncratic Brownian noise, the mean-field limit was recently proven in dimension $d=2$, without cut-off, in \cite{BreschJabinSolerVFP}.

For the ion model~\eqref{eq:vpme}, an analogous microscopic model to \eqref{eq:GeneralParticles} describes ions in the plasma as point charges, interacting with a thermalized electron distribution:
\be \label{eq:IonsSing}
\begin{cases}
\frac{\dd}{\dd t} X_i = V_i \\
\frac{\dd}{\dd t} V_i = \frac{1}{N} \sum_{j \neq i} K (X_i - X_j) - K \ast e^\Phi (X_i)
\end{cases}
\quad 
i \in \{1, \ldots, N \} ,
\ee
where 
\be
- \Delta \Phi =  \frac{1}{N} \sum_{j=1}^N \delta_{X_j} - e^\Phi .
\ee
As for the model for electrons, the interaction in this system becomes singular when $|X_i - X_j| \sim 0$. Unlike the electron model, and all models of the form \eqref{eq:GeneralParticles}, the interaction in \eqref{eq:IonsSing} depends \emph{nonlinearly} on the particle distribution.

The first result on the mean-field limit for \eqref{eq:vpme} was proven by Han-Kwan and Iacobelli \cite{HKI1} in dimension $d=1$.
In dimension $d=2,\,3$, \eqref{eq:vpme} was derived by the author and Iacobelli \cite{GPIMFQN} from
a system with regularised Coulomb interaction in which the ions are modelled as extended charges of radius order $r(N) \sim N^{-\frac{1}{d(d+2)}+\e}$ for any small $\e >0$,
using a deterministic method related to the techniques of Lazarovici \cite{Lazarovici}.
This left a gap between the available results for the ion and electron models.

In this article we narrow this gap by obtaining the limit for the ion model with an improved cut-off of $N^{-\frac{1}{d} + \e}$. To do this, we develop a probabilistic approach in the style of \cite{LazaroviciPickl}, whose novelty is that it is suitable for a nonlinear, implicitly defined interaction coupling.

\end{paragraph}

\begin{paragraph}{The Probabilistic Approach to the Mean-Field Limit for Nonlinear Coupling.}

In this article we develop a probabilistic approach to the mean-field limit that is suitable for particle systems with interaction forces satisfying a nonlinear coupling.
A key difficulty in applying the probabilistic approach of \cite{BoersPickl, LazaroviciPickl} in the nonlinear case is that the interaction force is not expressed directly in a form to which we can apply a (finite-dimensional) law of large numbers.

To understand this, consider first the electron model. Suppose that $(X_i, V_i)_{i=1}^N$ are independent random variables with identical distribution $f$.
Then the force exerted on particle $i$ by the other particles is $\frac{1}{N} \sum_{j \neq i}^N K(X_i - X_j)$. This expression is exactly a normalised sum of independent random variables to which the law of large numbers can be applied, so that we can expect, for large $N$,
\be
\frac{1}{N} \sum_{j \neq i}^N K(X_i - X_j) \approx \EE[ K(X_i - X) \vert X_i ] = K \ast \rho_f(X_i)
\ee
(here $X$ denotes a random variable independent of $X_i$ with density $f$).
A main element of the method of \cite{BoersPickl, LazaroviciPickl} is to quantify the rate of this convergence in probability as $N \to \infty$. Importantly, for linear couplings, this is required only at the finite collection of points $ X_1, \ldots, X_N $.

For the ion model, the force cannot be written directly in the form $\frac{1}{N} \sum_{j \neq i}^N \nabla \Psi (X_i - X_j)$. Instead, it is the solution of a nonlinear equation in which the particle distribution appears as a source term.
In order to use a probabilistic approach, we therefore require a (quantitative) law of large numbers holding in the sense of a function space rather than in a pointwise sense.

In this article, we derive the three-dimensional Vlasov-Poisson system for ions \eqref{eq:vpme} from a regularised approximation of \eqref{eq:IonsSing} with cut-off radius vanishing at rate no faster than $N^{-\frac{1}{d} + \e}$ as $N \to \infty$. The convergence holds on arbitrarily large time intervals, for as long as a sufficiently regular solution of \eqref{eq:vpme} exists.
To do this, we prove a concentration estimate for the convergence in $L^\infty$ sense of the convolution of an empirical measure with a $C^1$ function. Our estimate is well adapted to situations where the gradient becomes large at a point, allowing us to obtain the same asymptotics as Lazarovici and Pickl \cite{LazaroviciPickl} did for the electron model.
Additionally, we make use of the recently developed technique of nonlinearly anisotropically weighted kinetic `distances' (Iacobelli \cite{Iacobelli}) in order to streamline the proof.

\end{paragraph}

\subsection{Main Result}

The main result of this article, Theorem~\ref{thm:main}, is a rigorous derivation of the ionic Vlasov-Poisson system \eqref{eq:vpme} in the mean-field limit from a system of particles with cut-off interactions, where the scale of the cut-off vanishes with rate $r \geq b N^{- \frac{1}{d} + \e}$ for some constant $b>0$.
In this subsection we introduce the particle system in question and then state the main theorem.

\begin{paragraph}{Notation.}

Throughout the paper, we take the spatial domain to be $\TT^d$, the flat torus of dimension $d$.
Unless otherwise specified,
points $x \in \TT^d$ will by default be identified with their representatives in the fundamental domain $\left [ - \frac{1}{2}, \frac{1}{2} \right )^d$.
$\TT^d$ is equipped with the standard metric
\be \label{def:TorusDistance}
|x|_{\TT^d} : = \inf_{k \in \Z^d} |x + k| ,
\ee
where $| \cdot |$ is the usual Euclidean distance on $\R^d$.

For any $r > 0$ and $x \in \TT^d$, $B_r(x ; \TT^d)$ denotes the open ball in $\TT^d$ of radius $r$ and centre $x$.
$\overline{B_r}(x ; \TT^d)$ denotes the corresponding closed ball. Similarly $B_r(x ; \R^d)$ and $\overline{B_r}(x ; \R^d)$ denote respectively the open and closed balls in $\R^d$ of radius $r$ and centre $x \in \R^d$. Where the domain is clear from context, these notations are abbreviated to $B_r(x)$, $\overline{B_r}(x)$.

For a metric space $\M$, $\cP(\M)$ denotes the set of probability measures on $\mb{M}$, equipped with the topology of weak convergence of measures.

The Lebesgue spaces are denoted by $L^p(E)$ ($p \in [1, + \infty]$), where $E$ is any of the spaces $\TT^d$, $\R^d$ or $\TT^d \times \R^d$ equipped with the Lebesgue measure, or $L^p$ where $E$ is clear from context. $L^p_+(E)$ denotes the set of \emph{non-negative} $L^p(E)$ functions.

\end{paragraph}

\subsubsection{Particle Dynamics with Smoothed Interactions}

In this subsection, we set up the particle system from which we will derive the PDE \eqref{eq:vpme}. The system represents $N$ ions interacting with a background of thermalized electrons. However, rather than modelling the ions as point charges, we truncate the Coulomb interaction at scale $|x|_{\TT^d} \sim r > 0$ in order to remove the singularity that appears when two ions have a small spatial separation.
Later, $r$ will be taken to depend on the number of ions $N$, with $r \sim N^{- \frac{1}{d} + \e}$ for some $\e > 0$. 

To define the smoothed interaction,
 we first fix a mollifier function. Let $\chi \in C^\infty(\R^d ; [0 +\infty) )$ be a smooth, non-negative, radially symmetric function with compact support contained in $\overline{B_1}(0)$. Normalise $\chi$ to have total integral $\| \chi \|_{L^1} = 1$. Then, for any $r \in (0, \frac{1}{2} )$, we define a mollifier at scale $r$ on the torus $\TT^d$.

\begin{definition} \label{def:chi1}
For any $r \in (0, \frac{1}{2})$,
$\chi_r : \TT^d \to [0, +\infty)$ is the function defined by
\be \label{def:chi}
\chi_r (x) : = r^{-d} \chi \left ( \frac{x}{r} \right ) ,
\ee
where here each $x \in \TT^d$ is identified with its representative in $\left [-  \frac{1}{2}, \frac{1}{2} \right )^d$.
\end{definition}

Note that the restriction $r \in (0, \frac{1}{2})$ ensures that the support of $\chi_r$ is compactly contained in the open unit square $\left ( -\frac{1}{2}, \frac{1}{2} \right )^d$. Hence $\chi_r$ can be extended smoothly to a periodic function, so that Equation \eqref{def:chi} indeed specifies a well-defined $C^\infty$ function on $\TT^d$.

We use the mollifiers $\{ \chi_r \}_{r>0}$ to define a regularised approximation of the Coulomb force.

\begin{paragraph}{Notation for Particle Systems.}
The state of a particle system containing $N$ particles is represented by an $N$-tuple $({\bf X},{\bf V}) = (X_i, V_i )_{i=1}^N \in (\TT^d \times \R^d)^N$.
Each point $(X_i, V_i ) \in \TT^d \times \R^d $ is interpreted as the state of the $i$th particle in \emph{phase space}: $X_i \in \TT^d$ represents the particle's position, while $V_i \in \R^d$ represents the particle's velocity. All particles are assumed to share the same mass, so that $V_i$ is equivalent to the momentum of particle $i$, up to a global constant which we will set without loss of generality to one.

Any $N$-tuple $({\bf X},{\bf V})$ has a corresponding \emph{empirical measure} $\mu_{{\bf X},{\bf V}} \in \mc{P}(\TT^d \times \RR^d)$, which is a probability measure placing a Dirac mass (with uniform weight $\frac{1}{N}$) at the state $(X_i, V_i)$ of each particle:
\be
\mu_{{\bf X},{\bf V}} : = \frac{1}{N} \sum_{i = 1}^N \delta_{(X_i, V_i)} .
\ee
We will also use the notation $\mu_{\bf X} \in \mc{P}(\TT^d)$ for the empirical spatial measure:
\be
\mu_{{\bf X}} = \frac{1}{N} \sum_{i = 1}^N \delta_{X_i} .
\ee

\end{paragraph}

\begin{definition} \label{def:ParticleSystem}
Let $r \in (0, \frac{1}{2})$ and $N \in \N$.
Let $({\bf x}^{0,\,N},{\bf v}^{0,\,N}) = (x_i^{0,\,N}, v_i^{0,\,N} )_{i=1}^N \in (\TT^d \times \RR^d)^N$ be a given initial configuration. 
Then $({\bf X}^{r,\,N},{\bf V}^{r,\,N}) : [0, + \infty) \to (\TT^d \times \RR^d)^N$ denotes the solution of the following system of ODEs: 
writing $({\bf X}^{r,\,N}(t),{\bf V}^{r,\,N}(t)) = (X_i^{r,\,N}(t), V_i^{r,\,N}(t) )_{i=1}^N$,
\be \label{eq:ODE-N}
\begin{cases}
\frac{\dd}{\dd t} X_i^{r,\,N} = V_i^{r,\,N} \\
\frac{\dd}{\dd t} V_i^{r,\,N} = - \nabla \Phi_r[ \mu_{{\bf X}^{r,\,N}}] (X_i^{r,\,N}),
\end{cases}
\quad (X_i^{r,\,N}(0), V_i^{r,\,N}(0)) = (x_i^{0,\,N}, v_i^{0,\,N} ); 
\; \; \qquad  i = 1, \ldots, N,
\ee
where $\Phi_r [\mu_{{\bf X}^{r,\,N}}] : \TT^d \to \R$ denotes the solution of the equation
\begin{align}
- \Delta_x \Phi_r[\mu_{{\bf X}^{r,\,N}}] & = \chi_r \ast \mu_{{\bf X}^{r,\,N}} - e^{\Phi_r[\mu_{{\bf X}^{r,\,N}}]} \\
& = \frac{1}{N}  \sum_{j=1}^N \chi_r(\cdot  - X_j^{r,\,N}) - e^{\Phi_r[\mu_{{\bf X}^{r,\,N}}]} .
\end{align}
\end{definition}
\begin{remark}
The system \eqref{eq:ODE-N} is well-posed, globally in time, as can be shown using the methods of \cite[Lemma 6.2]{GPIWP}.
\end{remark}

\begin{remark} \label{rmk:abbrv}
Where $r = r(N)$, the notation $({\bf X}^{r,\,N},{\bf V}^{r,\,N})$ is abbreviated to $({\bf X}^{N},{\bf V}^{N})$.
\end{remark}

\subsubsection{Mean-Field Particle Dynamics}

The object of this article is to show that, for suitable choices of initial data and $r$ tending to zero at rate $r \sim N^{- \frac{1}{d} + \e}$, as $N$ tends to infinity the dynamics $({\bf X}^N, {\bf V}^N)$ of the interacting particle system \eqref{eq:ODE-N} closely approximate the \emph{mean-field} dynamics described by the ionic Vlasov-Poisson system. To do this, we seek to compare $({\bf X}^N, {\bf V}^N)$ to a solution of the following approximation of \eqref{eq:vpme}, in which the interaction has been regularised using $\chi_r$ in the same manner as used in the particle system in Definition~\ref{def:ParticleSystem}.
\be \label{eq:vpme-app}
\begin{cases}
\partial_t f_r + v \cdot \nabla_x f_r - \nabla_x \Phi_r [ \rho_{f_r} ] \cdot \nabla_v f_r = 0, \\
- \Delta_x \Phi_r[\rho_{f_r} ] = \chi_r \ast \rho_{f_r} - e^{\Phi_r[ \rho_{f_r} ] }, \\
f_r \vert_{t=0} = f_0 .
\end{cases} \qquad (x,v) \in \TT^d \times \R^d .
\ee
It can be shown, using the techniques of \cite{GPIWP}, that \eqref{eq:vpme-app} has a unique global solution $f_r$ for any initial probability measure $f_0 \in \cP(\TT^d \times \R^d)$.
$f_r$ may be represented as the \emph{pushforward} of $f_0$ along the flow $(Y^r, W^r)$ defined by the characteristic ODE of \eqref{eq:vpme-app}, which reads as follows: for all $(x,v) \in \TT^d \times \R^d$,
\be
\frac{\dd}{\dd t} Y^r(t;x,v) = W^r(t;x,v), \quad  \frac{\dd}{\dd t} W^r(t;x,v) = - \nabla \Phi_r [ \rho_{f_r}] (Y^r(t;x,v) ) ; \qquad (Y^r(0;x,v), W^r(0; x,v)) = (x, v ) .
\ee
That is, for any test function $\phi \in C_b(\TT^d \times \R^d)$,
\be \label{def:pushforward}
\int_{\TT^d \times \R^d} \phi(x,v) f_r(t, \dd x \dd v) = \int_{\TT^d \times \R^d} \phi \left (Y^r(t;x,v),W^r(t;x,v) \right ) f_0 (\dd x \dd v) \quad \text{for all } t \in [0,T).
\ee
However, to show that $f_r$ converges as $r$ tends to zero to a solution $f$ of \eqref{eq:vpme} typically requires additional hypotheses on $f_0$. We will assume the following uniform-in-$r$ estimate on the spatial density of $f_r$.

\begin{hyp} \label{hyp:f}
There exists $T > 0$, 
and a non-decreasing function $D: [0, T) \to [0,+\infty)$ such that
\be
\sup_{r \leq \frac{1}{2}} \sup_{s \in [0,t]} \| \rho_{f_r}(s, \cdot) \|_{L^\infty(\TT^d)} \leq D(t) \qquad \forall t \in [0,T) .
\ee
\end{hyp}

In dimension $d=2$ or $d=3$,
Assumption~\ref{hyp:f} holds with $T = + \infty$, for any $f_0 \in L^1 \cap L^\infty(\TT^d \times \R^d)$ satisfying, for some $k_0 > d$,
\be \label{hyp:f0}
f_0(x,v) \lesssim \frac{1}{1 + |v|^{k_0}}, \qquad \int_{\TT^d \times \RR^d} (1 + |v|^{k_0}) f_0(x,v) \dd x \dd v 
\ee
(this can be shown by combining a variant of \cite[Lemma 6.3]{GPIWP} with the techniques of \cite{GPIQN} for the improved moment condition).
In \cite{GPIWP} it was also shown that for $f_0$ satisfying \eqref{hyp:f0}, the Vlasov-Poisson system for ions \eqref{eq:vpme} has a global-in-time solution $f$, unique among the class of solutions with spatial density $\rho_f$ bounded in $L^\infty(\TT^d)$ locally uniformly in time, and that $f_r$ converges in the limit as $r$ tends to zero to $f$ in the sense of weak convergence of measures.

In fact, under Assumption~\ref{hyp:f} it is always the case that $f_r$ converges to the unique bounded density solution of \eqref{eq:vpme} on the time interval $[0,T)$, as we will show in Section~\ref{sec:VPMEConv} of this paper.
Moreover, the convergence holds in the sense that the characteristic flows converge uniformly (locally in time).

\begin{thm} \label{thm:VPMEconv}
Let $d \geq 2$.
Let $f_0 \in \cP(\TT^d \times \R^d)$ be fixed, and 
suppose that Assumption~\ref{hyp:f} holds. Then the ion Vlasov-Poisson system \eqref{eq:vpme} has a weak solution $f \in C \left ( [0, T) ; \mc{P}(\TT^d \times \R^d) \right )$.
This solution is unique among weak solutions of \eqref{eq:vpme} with initial datum $f_0$ and $\rho_f \in L^\infty_{\loc}( [0, T) ; L^\infty(\TT^d))$.

For each $r>0$, let $f_r$ denote the solution to the approximated ion Vlasov-Poisson system \eqref{eq:vpme-app} with the \emph{same} initial datum $f_0$. Then, as $r$ tends to zero, $f_r$ converges to $f$ in the sense of $C \left ( [0, T) ; \mc{P}(\TT^d \times \R^d) \right )$. 
Moreover, the corresponding characteristic flows converge uniformly on $\TT^d \times \R^d$, locally uniformly in time:
\be
\lim_{r \to 0} \| (Y^r - \overline Y, W^r - \overline W) \|_{L^\infty([0,t] \times \TT^d \times \R^d)} = 0, \quad  \text{for any} \; t \in [0, T) ,
\ee
where for any $(x,v) \in \TT^d \times \R^d$ the trajectory $(\overline Y(t ; x, v), \overline W(t ; x, v) )_{t \in [0,T)}$ satisfies the ODE
\be
\frac{\dd}{\dd t} \overline Y(t;x,v) = \overline W(t;x,v), \quad  \frac{\dd}{\dd t} \overline W (t;x,v) = - \nabla \Phi [ \rho_{f}] (\overline Y(t;x,v) ) ; \qquad (\overline Y(0;x,v), \overline  W (0; x,v)) = (x, v ) .
\ee

\end{thm}

In our main result (Theorem~\ref{thm:main}), we will compare $({\bf X}^N, {\bf V}^N)$ with $f_r$ in the limit $N \to +\infty$.
To do this, we use the method of e.g. \cite{Sznitman, BoersPickl, LazaroviciPickl}, 
introducing an \emph{auxiliary} $N$-particle system following the mean-field dynamics corresponding to $f_r$. The coupled and auxiliary particle systems are given the same initial data.

\begin{definition} \label{def:AuxParticles}
Let $0 < r < \frac{1}{2}$, and let $f_r$ be the solution of \eqref{eq:vpme-app}. For $t\geq0$, 
$({\bf Y}^{r,\,N}, {\bf W}^{r,\,N}) : [0, + \infty) \to (\TT^d \times \RR^d)^N$ denotes the solution of the following system of ODEs: 
writing $({\bf Y}^{r,\,N}(t), {\bf W}^{r,\,N}(t)) = (Y_i^{r,\,N}(t), W_i^{r,\,N}(t) )_{i=1}^N$,
\be \label{eq:ODE-aux}
\begin{cases}
\frac{\dd}{\dd t} Y_i^{r,\,N} = W_i^{r,\,N} \\
\frac{\dd}{\dd t} W_i^{r,\,N} = - \nabla \Phi_r [ \rho_{f_r}] (Y_i^{r,\,N}), 
\end{cases}
\quad (Y_i^{r,\,N}(0), W_i^{r,\,N}(0)) = (x_i^{0,\,N}, v_i^{0,\,N} ) ;
\; \; \qquad i \in \{ 1, \ldots, N \} , 
\ee
where $\Phi_r [\rho_{f_r}]$ is the solution of the equation
\begin{align}
- \Delta_x \Phi_r [\rho_{f_r}] & = \chi_r \ast \rho_{f_r} - e^{\Phi_r [\rho_{f_r}]} .
\end{align}
\end{definition}

\begin{remark} \label{rmk:ODEaux-decoupled}
Notice in particular that the ODE system \eqref{eq:ODE-aux} is \emph{decoupled}, such that the evolution of particle $i$ is not influenced by any other particle $j$, where $j \neq i$.
\end{remark}

\begin{remark} \label{rmk:AuxLaw}
Observe that $(Y_i^{r,\,N}, W_i^{r,\,N}) = (Y^r(t; x_i^{0,\,N}, v_i^{0,\,N}), W_i^r(t; x_i^{0,\,N}, v_i^{0,\,N}))$.
It then follows from the pushforward representation of $f_r$ \eqref{def:pushforward} that, if $(x_i^{0,\,N}, v_i^{0,\,N} )$ is chosen at random with law $f_0$, then $(Y_i^{r,\,N}(t), W_i^{r,\,N}(t))$ has law $f_r(t, \cdot)$, for any $t>0$.
\end{remark}

\begin{remark} 
As in Remark~\ref{rmk:abbrv},
where $r = r(N)$ the notation $({\bf Y}^{r,\,N},{\bf W}^{r,\,N})$ is abbreviated to $({\bf Y}^{N},{\bf W}^{N})$.
\end{remark}

\subsubsection{Selection of Initial Data}

We will choose the initial data points for the particle systems \eqref{eq:ODE-N} and \eqref{eq:ODE-aux} randomly so that each $(x_i^{0, \, N}, v_i^{0, \, N} )$ has law $f_0$. Furthermore, we will take them to be statistically \emph{independent}. 

To fix notation, let $(\Omega, \cF, \PP)$ be a probability space. For $N \in \N$ given and $i = 1, \ldots, N$, let $(x_i^{0, \, N}, v_i^{0, \, N}) : \Omega \to \TT^d \times \R^d$ be independent random variables, each with law induced by the probability density $f_0$. That is, under the measure
$\PP$,
\be
({\bf x}^{0, \, N}, {\bf v}^{0, \, N}) = (x_i^{0, \, N}, v_i^{0, \, N})_{i=1}^N \sim f_0^{\otimes N} .
\ee
For all $t \in (0, \frac{1}{2})$, the $N$-tuples $({\bf X}^{r, \,N}, {\bf V}^{r, \,N})$, $({\bf Y}^{r, \, N}, {\bf W}^{r, \,N})$ defined in Definitions~\ref{def:ParticleSystem} and \ref{def:AuxParticles} are then also random variables. Moreover, by Remark~\ref{rmk:ODEaux-decoupled}, the independence of $(Y_i^{r, \, N}(t), W_i^{r, \, N}(t))_{i=1}^N$ is preserved along the evolution. Hence,
for any $t \in [0, T)$ the random $N$-tuple $({\bf Y}^{r, \, N}(t), {\bf W}^{r, \, N}(t))$
has law
\be
({\bf Y}^{r, \, N}(t), {\bf W}^{r, \, N}(t)) \sim f_r(t, \cdot)^{\otimes N} .
\ee
In other words, $({\bf Y}^{r, \, N}(t), {\bf W}^{r, \, N}(t))$ is distributed like $N$ independent samples from $f_r(t, \cdot)$.

\subsubsection{Mean-Field Limit for the Ionic Vlasov-Poisson System}

Our main result states that the coupled particles $({\bf X}^N, {\bf V}^N)$ closely approximate the auxiliary particles $({\bf Y}^N, {\bf W}^N)$ with high probability as $r$ tends to zero with $r \sim N^{- \frac{1}{d} + \e}$.
We will use the following notation for the maximal distance between two (time-dependent) $N$-tuples, representing the dynamics of two particle systems with the same number of particles $N$.

\begin{paragraph}{Notation: Comparison of $N$-tuples.}
Let $({\bf X},{\bf V}) \in (\TT^d \times \R^d)^N$ be an $N$-tuple of points in $\TT^d \times \R^d$.
Then $| {\bf X}|_\infty$ and $| {\bf V}|_\infty$ denote the norms
\be
| {\bf X}|_\infty : = \max_i |X_i |_{\TT^d} , \qquad  | {\bf V}|_\infty : = \max_i |V_i | .
\ee

Now suppose that $({\bf X},{\bf V}): [0, T) \to (\TT^d \times \R^d)^N$ is a continuous function taking values in the space of $N$-tuples. For any $t \in [0, T)$, we use the notation $\| \cdot \|_{L^\infty([0,t])}$ to denote the supremum norm in time:
\be
\| {\bf X} \|_{L^\infty([0,t])} : = \sup_{s \in [0,t]} | {\bf X} (s) |_{\infty} , \qquad  \| {\bf V}  \|_{L^\infty([0,t])}  : = \sup_{s \in [0,t]}  | {\bf V} (s) |_\infty .
\ee

\end{paragraph}

We are now ready to state the main theorem of this article.

\begin{thm} \label{thm:main}
Let $d \geq 2$ and $f_0 \in \cP(\TT^d \times \R^d)$ be given, and suppose that
Assumption~\ref{hyp:f} is satisfied.
Assume that $r = r(N)$ tends to zero as $N$ tends to infinity with rate satisfying, for some $\e \in ( 0, \frac{1}{d})$ and $b>0$, $r = r(N) \geq b N^{- \frac{1}{d} + \e}$.

For $N\geq 1$, let the following hold:
\begin{itemize}
\item The points $(x_i^{0, \, N}, v_i^{0, \, N})_{i=1}^N$ are drawn independently from $f_0$, the initial datum for \eqref{eq:vpme}:
\be
(x_i^{0, \, N}, v_i^{0, \, N})_{i=1}^N \sim f_0^{\otimes N} ;
\ee
\item $({\bf X}^N,{\bf V}^N)$ is the solution of the particle system with cut-off interactions \eqref{eq:ODE-N}, with initial data $(x_i^{0, \, N}, v_i^{0, \, N})_{i=1}^N$.
\item $({\bf Y}^N,{\bf W}^N)$ is the solution of the auxiliary system \eqref{eq:ODE-aux}, with the same initial data $(x_i^{0, \, N}, v_i^{0, \, N})_{i=1}^N$.
\end{itemize}

Then, for all $t \in [0, T)$, for any $\sigma \in (0, d \epsilon )$ there exists a positive constant $C>0$ (which depends on $t$, $D(t)$, $d$, $\e$, $\sigma$, $b$, and the choice of mollifier $\chi$), such that 
\be
\PP \left (  \|{ \bf X}^N - {\bf Y }^N \|_{L^\infty([0,t])}  + \| {\bf V }^N - {\bf W}^N\|_{L^\infty([0,t])}  \geq r \right ) \leq  C e^{- C N^\sigma} . 
\ee
\end{thm}

\begin{remark}
As explained in \cite{LazaroviciPickl}, Theorem~\ref{thm:main} implies convergence in probability of the empirical measure $\mu_{({\bf X}^N, {\bf V}^N)}$ in Monge-Kantorovich-Wasserstein distances $W_p$ for $p \in [1, +\infty)$, under suitable moment conditions on the initial datum $f_0$. This follows from the $W_p$ concentration estimates of Fournier and Guillin \cite{FournierGuillin}.
\end{remark}

When combined with Theorem~\ref{thm:VPMEconv}, Theorem~\ref{thm:main} implies that the dynamics of $({\bf X}^N, {\bf V}^N)$ converge to those of the Vlasov-Poisson system for ions \eqref{eq:vpme} (with singular interaction) as $N$ tends to infinity.
For example, in three dimensions we may deduce the following corollary.

\begin{cor} \label{cor:3d} 
Let $d=3$. Suppose $f_0 \in L^1_+ \cap L^\infty (\TT^3 \times \R^3)$ and, for some $k_0 > 3$, $(1 + |v|^{k_0}) f_0 \in L^1 \cap L^\infty(\TT^3 \times \R^3)$.
Let $f$ denote the unique solution of \eqref{eq:vpme} with initial datum $f_0$ and $\rho_f \in L^\infty( [0, + \infty) ; L^\infty(\TT^3))$, and let $( \overline{Y}, \overline{W})$ denote the corresponding characteristic flow.
Let $r = r(N) \to 0$ as $N \to + \infty$ with, for some $\e \in ( 0, \frac{1}{3})$ and $b>0$, $r = r(N) \geq b N^{- \frac{1}{3} + \e}$.

For all integers $N \geq 1$, let $(x_i^{0, \, N}, v_i^{0, \, N})_{i=1}^N$ have distribution $f_0^{\otimes N}$ under $\PP$.
Let $({\bf X}^N,{\bf V}^N)$ be the solution of the particle system with cut-off interactions \eqref{eq:ODE-N}, with initial data $(x_i^{0, \, N}, v_i^{0, \, N})_{i=1}^N$ (see Definition~\ref{def:ParticleSystem}).
Let $(\overline{{\bf Y }}^N, \overline{{\bf W }}^N)$ be the mean-field dynamics induced by $f$, with the same initial data---i.e.,
\be
(\overline{Y_i}^N, \overline{W_i}^N) = \left ( \overline{Y} (x_i^{0, \, N}, v_i^{0, \, N}) , \, \overline{W}(x_i^{0, \, N}, v_i^{0, \, N}) \right ) \qquad \text{for all} \; i \in \{ 1, \ldots, N \} .
\ee

Then, $\PP$-almost surely, 
\be
\lim_{N \to 0} \|{ \bf X}^N - \overline{{\bf Y }}^N \|_{L^\infty([0,t])}  + \| {\bf V }^N - \overline{{\bf W}}^N\|_{L^\infty([0,t])} = 0 \quad \text{for all} \; t > 0.
\ee

\end{cor}

\subsection{Structure of the Paper}

The remainder of the article is structured as follows. 
Section~\ref{sec:potential} contains estimates for the electrostatic potential $\Phi$---in particular, a new quantitative stability estimate for the convergence of the electric field in $L^\infty$.
Section~\ref{sec:LLN} focuses on the $L^\infty$ law of large numbers for convolutions with empirical measures induced by a small perturbation of an $N$-tuple of independent and identically distributed random variables.
Then, in Section~\ref{sec:MFL} we combine the results of Sections~\ref{sec:potential} and \ref{sec:LLN} to
prove Theorem~\ref{thm:main}, using an approach based on the nonlinear kinetic `distances' introduced in \cite{Iacobelli}.
The article concludes with the proof of Theorem~\ref{thm:VPMEconv} in Section~\ref{sec:VPMEConv}.

\begin{paragraph}{Notation.}

In inequalities of the form $A \leq C B$, $C>0$ denotes a generic positive constant, independent of any relevant parameters, which may change from line to line without comment. Subscripts such as $C_\alpha$ are used to indicate cases in which $C_r$ depends on a parameter $\alpha$ (and no others) in a manner that we will not write explicitly.
For ease of reading, the symbol $\lesssim$ is used to abbreviate inequalities like $A \leq C B$ to $A \lesssim B$, where the implicit constant $C$ is generic as described. Similarly $\lesssim_\alpha$ indicates that the implicit constant $C_\alpha$ depends on the parameter $\alpha$.

We abbreviate the maximum and minimum functions using the symbols $\vee$ and $\wedge$, respectively. Thus $a \vee b = \max\{ a, b\}$ and $a \wedge b = \min\{ a , b \}$. The positive part of $a$ is denoted by $a_+ = a \vee 0$.

\end{paragraph}

\section{Analysis of the Potential} \label{sec:potential}

This section focuses on estimates for the electrostatic potentials in the ODE systems \eqref{eq:ODE-N} and \eqref{eq:ODE-aux}.
In either case, the potential is a solution of the Poisson-Boltzmann equation \eqref{eq:potential}, for some source $\rho$.
For the applications we have in mind, we expect to have $\rho \in L^\infty_+(\TT^d)$, either by Assumption~\ref{hyp:f} in the mean-field case, or in general for regularised systems where $\rho$ is of the form $\chi_r \ast \mu$ for a probability measure $\mu$.
We first summarise the following regularity and integrability estimates for $\Phi$; see \cite{GPIWP, GPIQN} for further details and proofs.

\begin{prop} \label{prop:potential}
Let $\rho \in L^\infty_+(\TT^d)$.
There exists a unique solution $\Phi \in (W^{1,2} \cap L^\infty)(\TT^d)$ of the equation
\be \label{eq:potential}
-  \Delta \Phi = \rho - e^\Phi .
\ee
$\Phi$ satisfies the following estimate, for any $p \in [1, +\infty]$:
\be \label{est:rhoLp}
\| e^\Phi \|_{L^p} \leq \| \rho  \|_{L^p} .
\ee
Moreover $\Phi \in W^{2,p}(\TT^d)$ for all $p \in (1, + \infty)$, and $\nabla \Phi$ is log-Lipschitz with the following bound: for all $x,y \in \TT^d$,
\be
|\nabla \Phi (x) - \nabla \Phi (y)| \leq C_d  \| \rho \|_{L^\infty} |x - y |_{\TT^d} \log |x - y|_{\TT^d}^{-1}  ,
\ee
where $C_d > 0$ is a uniform constant depending on the dimension $d$ only.
\end{prop}

The main object of this section is to prove an estimate controlling in $L^\infty(\TT^d)$ the difference between electric fields induced by different densities. Existing estimates of this kind for the ion model \cite{GPIWP, GPIMFQN, GPIQN} are of $L^2$ type and depend on the 2-Wasserstein distance between the densities.
In the present setting, we want to control the deviation between the two particle systems in the uniform distance $\| {\bf X }^N - { \bf Y }^N \|_{L^\infty} + \| {\bf V }^N - { \bf W }^N \|_{L^\infty}$, and therefore require $L^\infty$-type estimates for the electric field.

In order to state the new estimate (Proposition~\ref{prop:PotentialStability}), we first need to introduce the Green's function of the Laplacian on the $d$-dimensional flat torus $\TT^d$, which will be denoted by $G$. 
$G \in \mc{D}'(\TT^d)$ is a distribution such that
\be
- \Delta G = \delta_0 - 1 \qquad \text{on } \TT^d , \qquad \langle G, 1 \rangle = 0 .
\ee
We let $K : = - \nabla_x G$. $G$ and $K$ are $C^\infty$ functions on $\TT^d \setminus \{ 0 \}$, and the singularity at $x=0$ has the following behaviour: there exists a $C^\infty(\TT^d)$ function $G_0$ such that $G = G_0 + G_1$ and for all $x$ such that $|x|_{\TT^d} < \frac{1}{4}$
\be \label{def:SingKernel}
G_1(x) = \begin{cases}
- c_2 \log { |x|_{\TT^d} } & d = 2, \\
 \frac{c_d}{|x|_{\TT^d}^{d-2}} & d \geq 3,
 \end{cases}
\ee
for some positive constants $c_d > 0$ (see \cite{Titchmarsh}). In particular $K_1 = - \nabla_x G_1$ satisfies
\be \label{def:K1}
K_1 =  - c_d ' \frac{x}{|x|_{\TT^d}^{d}}, \qquad \text{for all} \; \, |x|_{\TT^d} < \frac{1}{4} .
\ee

\begin{prop} \label{prop:PotentialStability}
For $i=1,2$, let
$\Phi_i$ satisfy \eqref{eq:potential}
with the density $\rho_i \in L^\infty_+(\TT^d)$.
Then
\be
\|\nabla \Phi_1 - \nabla \Phi_2 \|_{L^\infty(\TT^d)} \leq \exp \left (C_d (1 + \max_i \| \rho_i \|_{L^\infty(\TT^d)}) \right ) \| K \ast \rho_1 - K \ast \rho_2 \|_{L^\infty(\TT^d)} ,
\ee
for some constant $C_d > 0$ depending on the dimension $d$ only.

\end{prop}

To prove Proposition~\ref{prop:PotentialStability}, we first prove two preliminary lemmas. The first is an $L^p$ stability estimate for the nonlinearity $e^\Phi$.

\begin{lemma} \label{lem:NLStability}
Let $\Phi_1$ and $\Phi_2$ satisfy the assumptions of Proposition~\ref{prop:PotentialStability}.
Suppose that there exist constants $D_0, \, D_1 > 0$ such that
\be
D_0 \leq e^{\Phi_1}, \; e^{\Phi_2} \leq D_1 \qquad \text{for all } x \in \TT^d .
\ee
Then, for any $p \in (d, + \infty)$,
\be
\| e^{\Phi_1} - e^{\Phi_2} \|_{L^p} \leq C_p \, D_1 D_0^{-1/2} \|  K \ast (\rho_1 - \rho_2) \|_{L^p} ,
\ee
where $C_p > 0$ is a constant depending on $p$ only.

\end{lemma}

\begin{remark}
Note that, by Proposition~\ref{prop:potential}, we can always choose $D_1 = \max_i \| \rho_i \|_{L^\infty}$.
\end{remark}

\begin{proof}
The difference $\Phi_1 - \Phi_2$ satisfies the equation
\be \label{eq:PBDiff-prelim}
-  \Delta (\Phi_1 - \Phi_2) = \rho_1 - \rho_2 - (e^{\Phi_1} - e^{\Phi_2}) .
\ee
Since $\int_{\TT^d} (\rho_1 - \rho_2) \dd x = 0$, we may rewrite this in the form
\be \label{eq:PBDiff}
-  \Delta (\Phi_1 - \Phi_2) = \div \left [ K \ast (\rho_1 - \rho_2) \right ] - (e^{\Phi_1} - e^{\Phi_2}) .
\ee

We multiply the equation \eqref{eq:PBDiff} by the function $(\Phi_1 - \Phi_2)|\Phi_1 - \Phi_2|^{p-2}$ and integrate to obtain, after an integration by parts, the equality
\begin{multline}
\int_{\TT^d}  (e^{\Phi_1} - e^{\Phi_2}) (\Phi_1 - \Phi_2) |\Phi_1 - \Phi_2|^{p-2} \dd x +  \int_{\TT^d} \nabla (\Phi_1 - \Phi_2) \cdot  \nabla \left [ (\Phi_1 - \Phi_2) |\Phi_1 - \Phi_2|^{p-2} \right ] \dd x \\
= - \int_{\TT^d} K \ast (\rho_1 - \rho_2) \cdot  \nabla \left [ (\Phi_1 - \Phi_2) |\Phi_1 - \Phi_2|^{p-2} \right ]  \dd x .
\end{multline}

Next, note that
\begin{align}
 (e^{\Phi_1} - e^{\Phi_2}) (\Phi_1 - \Phi_2) |\Phi_1 - \Phi_2|^{p-2} & \geq |\Phi_1 - \Phi_2|^{p} \, \int_0^1 e^{\lambda \Phi_1 + (1- \lambda) \Phi_2} \dd \lambda \\
 & \geq D_0 |\Phi_1 - \Phi_2|^{p} .
\end{align}
Using the relations
\begin{align}
& \nabla \left [  (\Phi_1 - \Phi_2)|\Phi_1 - \Phi_2|^{p-2} \right ] = (p-1) |\Phi_1 - \Phi_2|^{p-2} \nabla (\Phi_1 - \Phi_2), \\
& |\Phi_1 - \Phi_2 |^{p/2-1}  \nabla (\Phi_1 - \Phi_2) = \frac{2}{p} \nabla \left ( (\Phi_1 - \Phi_2) |\Phi_1 - \Phi_2 |^{p/2 -1} \right ),
\end{align}
we observe that we have obtained the inequality
\begin{multline}
 D_0 \int_{\TT^d}|\Phi_1 - \Phi_2|^{p} \dd x + \frac{4(p-1)}{p^2} \int_{\TT^d} \left | \nabla \left ( (\Phi_1 - \Phi_2) |\Phi_1 - \Phi_2 |^{p/2 -1} \right ) \right |^2 \dd x \\
\leq \frac{2(p-1)}{p} \int_{\TT^d} \left | K \ast (\rho_1 - \rho_2) \right | \, |\Phi_1 - \Phi_2|^{p/2-1} \, \left | \nabla \left ( (\Phi_1 - \Phi_2) |\Phi_1 - \Phi_2 |^{p/2 -1} \right ) \right |\dd x .
\end{multline}

Then, by H\"older's inequality,
\begin{multline}
D_0 \| \Phi_1 - \Phi_2 \|_{L^p}^p+ \frac{4(p-1)}{p^2} \left \| \nabla \left ( (\Phi_1 - \Phi_2) |\Phi_1 - \Phi_2 |^{p/2 -1} \right )  \right \|_{L^2}^2 \\
\leq \frac{2(p-1)}{p} \|  K \ast (\rho_1 - \rho_2) \|_{L^p} \left \| \Phi_1 - \Phi_2 \right \|_{L^p}^{p/2 - 1} \left \| \nabla \left ( (\Phi_1 - \Phi_2) |\Phi_1 - \Phi_2 |^{p/2 -1} \right )  \right \|_{L^2} .
\end{multline}
By Young's inequality for products (with exponents $2$ and $2$),
\begin{multline}
D_0 \| \Phi_1 - \Phi_2 \|_{L^p}^p+ \frac{4(p-1)}{p^2} \left \| \nabla \left ( (\Phi_1 - \Phi_2) |\Phi_1 - \Phi_2 |^{p/2 -1} \right )  \right \|_{L^2}^2 \\
\leq (p-1) \|  K \ast (\rho_1 - \rho_2) \|_{L^p}^2 \left \| \Phi_1 - \Phi_2 \right \|_{L^p}^{p - 2} + \frac{2(p-1)}{p^2} \left \| \nabla \left ( (\Phi_1 - \Phi_2) |\Phi_1 - \Phi_2 |^{p/2 -1} \right )  \right \|_{L^2}^2
\end{multline}
and hence
\be
\| \Phi_1 - \Phi_2 \|_{L^p} \leq C_p D_0^{- 1/p} \|  K \ast (\rho_1 - \rho_2) \|_{L^p}^{2/p} \left \| \Phi_1 - \Phi_2 \right \|_{L^p}^{1 - 2/p} .
\ee
By another application of Young's inequality for products (exponents $\frac{p}{2}$ and $\frac{p}{p-2}$), we obtain 
\be
\| \Phi_1 - \Phi_2 \|_{L^p} \leq C_p D_0^{- 1/2} \|  K \ast (\rho_1 - \rho_2) \|_{L^p} .
\ee

Finally, since
\be
| e^{\Phi_1} - e^{\Phi_2} | \leq | \Phi_1 - \Phi_2 | \sup_{\lambda \in (0,1), \; x \in \TT^d} e^{\lambda \Phi_1(x) + (1- \lambda) \Phi_2(x)} \leq D_1  | \Phi_1 - \Phi_2 |,
\ee
we have
\be
\| e^{\Phi_1} - e^{\Phi_2} \|_{L^p} \leq C_p D_1 D_0^{-1/2} \|  K \ast (\rho_1 - \rho_2) \|_{L^p} .
\ee

\end{proof}

Next, we will prove a lower bound $D_0$ on the nonlinearity $e^\Phi$. 

\begin{lemma}
\label{lem:LB}
Let $\rho \in L^p_+(\TT^d)$ for $p > \frac{d}{2}$. 
Let $\Phi \in W^{1,2} \cap L^\infty$ satisfy the Poisson-Boltzmann equation \eqref{eq:potential} with density $\rho$.
Then there exists $C_{p,d} >0$, depending on $p$ and $d$ only, such that
\be
e^\Phi \geq \exp{(- C_{p,d} \| \rho  \|_{L^{p}})} > 0.
\ee
\end{lemma}

\begin{proof}
We may use the elliptic equation \eqref{eq:potential} to control the deviation of $\Phi$ from its average value
\be
\langle \Phi \rangle : = \frac{1}{|\TT^d|} \int_{\TT^d} \Phi \dd x =  \int_{\TT^d} \Phi \dd x ;
\ee
the second inequality holds because we have normalised $\TT^d$ to have unit volume.
In particular
\be
\Phi - \langle \Phi \rangle = G \ast ( \rho - e^\Phi ).
\ee
We now estimate this convolution in $L^\infty(\TT^d)$.

Using the decomposition $G = G_0 + G_1$, and by the form \eqref{def:SingKernel} of $G_1$, we have $G \in L^q(\TT^d)$ for all $q \in [1, \frac{d}{d-2})$ ($q \in [1, +\infty)$ for $d=2$). Since $p'$, the Hölder conjugate exponent of $p$, lies in this range, we may estimate
\be
\| \Phi - \langle \Phi \rangle \|_{L^\infty} \leq \| G \|_{L^{p'}} \|  \rho - e^\Phi \|_{L^p} .
\ee
Hence, by Proposition~\ref{prop:potential},
\be
\| \Phi - \langle \Phi \rangle \|_{L^\infty} \leq C_{p,d} \|  \rho \|_{L^p} ,
\ee
where $C_{p,d}>0$ is a positive constant depending on $p$ and $d$ only.
We rewrite this inequality as
\be
\langle \Phi \rangle - C_{p,d} \| \rho  \|_{L^{p}} \leq \Phi \leq \langle \Phi \rangle + C_{p,d} \| \rho  \|_{L^{p}} .
\ee
It remains only to control the constant $\langle \Phi \rangle$ from below.

To do this, we use the fact that the nonlinearity $e^\Phi$ has unit integral.
For all $x \in \TT^d$,
\be
e^{\Phi(x)}  \leq \exp{(\langle \Phi \rangle + C_{p,d} \| \rho  \|_{L^{p}} )} .
\ee
Integrating this inequality over all $x \in \TT^d$, we find that
\be
1 = \int_{\TT^d} e^\Phi   \dd x \leq  \exp{(\langle \Phi \rangle + C_{p,d} \| \rho  \|_{L^{p}} )} .
\ee
Taking logarithms, we obtain
\be
\langle \Phi \rangle \geq - C_{p,d} \| \rho  \|_{L^{p}} .
\ee
Hence
\be
 \Phi  \geq - C_{p,d} \| \rho  \|_{L^{p}} \qquad \text{for all } x \in \TT^d.
\ee
Taking exponentials completes the proof. 

\end{proof}

We are now able to complete the proof of Proposition~\ref{prop:PotentialStability}.

\begin{proof}[Proof of Proposition~\ref{prop:PotentialStability}]
We have 
\be
\nabla \Phi_1 - \nabla \Phi_2 = K \ast (\rho_1 - \rho_2) - K \ast (e^{\Phi_1} - e^{\Phi_2}) ,
\ee
and hence
\be
\| \nabla \Phi_1 - \nabla \Phi_2 \|_{L^\infty} = \| K \ast (\rho_1 - \rho_2) \|_{L^\infty}  + \| K \ast (e^{\Phi_1} - e^{\Phi_2}) \|_{L^\infty} .
\ee
To estimate the second term, choose any $p \in (d, + \infty)$. Then, since by \eqref{def:SingKernel} $K \in L^{p'}$,
\be
\| K \ast (e^{\Phi_1} - e^{\Phi_2}) \|_{L^\infty} \leq \| K \|_{L^{p'}} \| e^{\Phi_1} - e^{\Phi_2} \|_{L^p} .
\ee
Then, by Lemma~\ref{lem:NLStability},
\be
\| K \ast (e^{\Phi_1} - e^{\Phi_2}) \|_{L^\infty} \leq C_p \, D_1 D_0^{-1/2} \|  K \ast (\rho_1 - \rho_2) \|_{L^p} ,
\ee
where $D_0, D_1$ are constants such that, for $i=1,2$,
\be
D_0 \leq e^{\Phi_i} \leq D_1 .
\ee

By Proposition~\ref{prop:potential}, we may take $D_1 = \max_i \| \rho \|_{L^\infty}$. By Lemma~\ref{lem:LB}, we may take $D_0$ of the form $D_0 = e^{-C_d \max_i \| \rho_i \|_{L^\infty}}$.
Hence
\begin{align}
\| K \ast (e^{\Phi_1} - e^{\Phi_2}) \|_{L^\infty} & \leq C_p \, e^{ C_d \max_i \| \rho_i \|_{L^\infty}} (\max_i \| \rho_i \|_{L^\infty}) e^{ C_d \max_i \| \rho_i \|_{L^\infty}} \|  K \ast (\rho_1 - \rho_2) \|_{L^p} \\
& \leq  e^{ C_{d,p} (1 +  \max_i \| \rho_i \|_{L^\infty})} \|  K \ast (\rho_1 - \rho_2) \|_{L^p} \\
& \leq e^{ C_{d,p} (1 +  \max_i \| \rho_i \|_{L^\infty})} \|  K \ast (\rho_1 - \rho_2) \|_{L^\infty}
\end{align}
as required.

\end{proof}

\section{Uniform Approximation Estimates for Empirical Means} \label{sec:LLN}

One of the key techniques in \cite{BoersPickl, LazaroviciPickl} is to show convergence of the microscopic force field using a law of large numbers argument. 
In the models considered in these works, the interaction force is of convolution form $- \nabla \Psi \ast \mu$: the force between $N$ particles with positions $\{ Z_i \}_{i=1}^N$ is of the form $- \frac{1}{N} \sum_{j\neq i}^N \nabla \Psi (Z_i - Z_j)$ at particle $i$. If $\{ Z_i \}_{i=1}^N$ are independent with identical distribution $\rho_Z$, then by the law of large numbers this converges almost surely as $N$ tends to infinity to $g \ast \rho_Z(Z_i)$. Repeating this argument for all $i = 1, \ldots, N$ shows that the interaction force experienced by all the particles converges in the limit.

The model for ions differs in that the interaction potential depends \emph{nonlinearly} on the particle density, and in particular is not of the form $\frac{1}{N} \sum_{j\neq i}^N g(Z_i - Z_j)$ for any function $g$.
By Proposition~\ref{prop:PotentialStability}, it would suffice to show the convergence of $K \ast \chi_r \ast \mu_{{\bf Z}}$ and uniform bounds on $\chi_r \ast \mu_{{\bf Z}}$. However, we require this convergence and boundedness not only at a finite number of points $\{ Z_i \}_{i=1}^N$ but for the whole function in $L^\infty(\TT^d)$ sense.

The goal of this section is to prove a uniform law of large numbers of this type, with quantitative bounds for the rate of the convergence in probability. 
Laws of large numbers and concentration inequalities for Banach-space valued random variables have been widely studied in the probability and statistics literature---see for example \cite{LedouxTalagrand, BoucheronLugosiMassart}.
Here we will prove the result we require directly, using local Lipschitz estimates to obtain similar asymptotics as the cut-off is removed to those found in \cite{LazaroviciPickl}.

\subsection{Local Lipschitz Property}

Another important component of the methods of \cite{BoersPickl, LazaroviciPickl} is the use of local Lipschitz estimates. For example, as observed in \cite[Lemma 6.3]{LazaroviciPickl}, for the Coulomb kernel on $\R^d$ truncated at length scale $r$ (denoted here by $\tilde K_r$) the following estimate holds: there exists a non-negative function $\tilde L_r$ with $\| \tilde L_r \|_{L^1} \sim |\log r|$ such that
\be \label{ex:locLip}
| \tilde K_r (x) - \tilde K_r(y) | \leq \tilde L_r (y) |x-y| \quad \text{for all} \; x,y \in \R^d \; \text{such that} \; |x-y| < r. 
\ee
The key point in this inequality is that the function $\tilde L_r$ is evaluated only at the point $y$. 
In order to prove convergence in the case of nonlinear coupling, we will need a version of property \eqref{ex:locLip} for several functions. 
For convenience, we will use the following terminology.

\begin{definition} \label{def:localLip}
Let $g: \TT^d \to \R^m$ be a continuous function. $h : \TT^d \to [0, + \infty)$ is an \emph{$r$-local Lipschitz modulus} for $g$ if, for all $x,y \in \TT^d$ such that $|x-y|_{\TT^d} < r$,
\be \label{eq:ModOfCont}
\left | g(x) - g(y) \right | \leq h(y) |x-y|_{\TT^d} .
\ee
\end{definition}
\begin{remark}
Notice that any $r$-local Lipschitz modulus is an $r'$-local Lipschitz modulus for any $0 < r' < r$.
\end{remark}

Once again, we emphasise that in this definition the modulus $h$ is evaluated at the point $y$ only, rather than at an intermediate point between $x$ and $y$, as a standard mean value argument would give for any $C^1$ function.
Nevertheless, any $g \in C^1$ does have an $r$-local Lipschitz modulus, for example the function $h$ defined by
\be \label{C1-Lip}
h(y) : = \sup_{z \in \overline{B_r}(y)} |\nabla g(z)| .
\ee
For functions defined on $\RR^d$, this is an immediate consequence of the mean value theorem. 
We can extend the result to functions defined on the torus $\TT^d$, since it is always possible to interpolate between $x$ and $y$ along a path of length $|x-y|_{\TT^d}$ wholly contained within $B_r(y)$. 

In \cite{BoersPickl, LazaroviciPickl}, these local Lipschitz estimates were used to show \emph{pointwise} convergence of functions of the form $g \ast \mu_{{\bf Z}}$. In the case of the ion model we will need to show that convergence in fact holds in the stronger sense of $L^\infty(\TT^d)$. To achieve this,
we will additionally use a `second-order' estimate: that is, an $r$-local Lipschitz estimate for the modulus $h$ defined by \eqref{C1-Lip}. 

\begin{lemma} \label{lem:l} 
Assume that $r  \in (0, \frac{1}{4})$.
Let $g \in C^2(\TT^d)$. Define $h : \TT^d \to [0, + \infty)$ by 
\be \label{def:h}
h(y) : = \sup_{z \in \overline{B_r}(y)} |\nabla g(z)|  \qquad \text{ for all } y \in \TT^d.
\ee
Then $h$ is Lipschitz, and the function $l : \TT^d \to [0, + \infty)$ defined by
\be \label{def:l}
l(y) : =  \sup_{z \in \overline{B_{2r}}(y)} |\nabla^2 g(z)|  \qquad \text{ for all } y \in \TT^d.
\ee
is an $r$-local Lipschitz modulus for $h$.
\end{lemma}
\begin{remark}
The argument can be iterated to obtain that, for any $k \in \mathbb{N}$, the function
\be
\sup_{z \in \overline{B_{kr}}(y)} |\nabla^k g(z)| 
\ee 
is Lipschitz with $r$-local modulus
\be
\sup_{z \in \overline{B_{(k+1)r}}(y)} |\nabla^{k+1} g(y + z)| .
\ee
\end{remark}

\begin{remark}
A similar result holds for functions defined on $\R^d$, without the restriction $r < \frac{1}{4}$.
\end{remark}

We will need the following technical lemma, which gives a convexity property for sufficiently small balls in $\TT^d$ (see Figure~\ref{fig:toruspaths}).

\begin{figure}[h]
	\centering
	
	\begin{subfigure}[b]{0.3 \textwidth}
		\centering
		\includegraphics[width = 0.87 \textwidth]{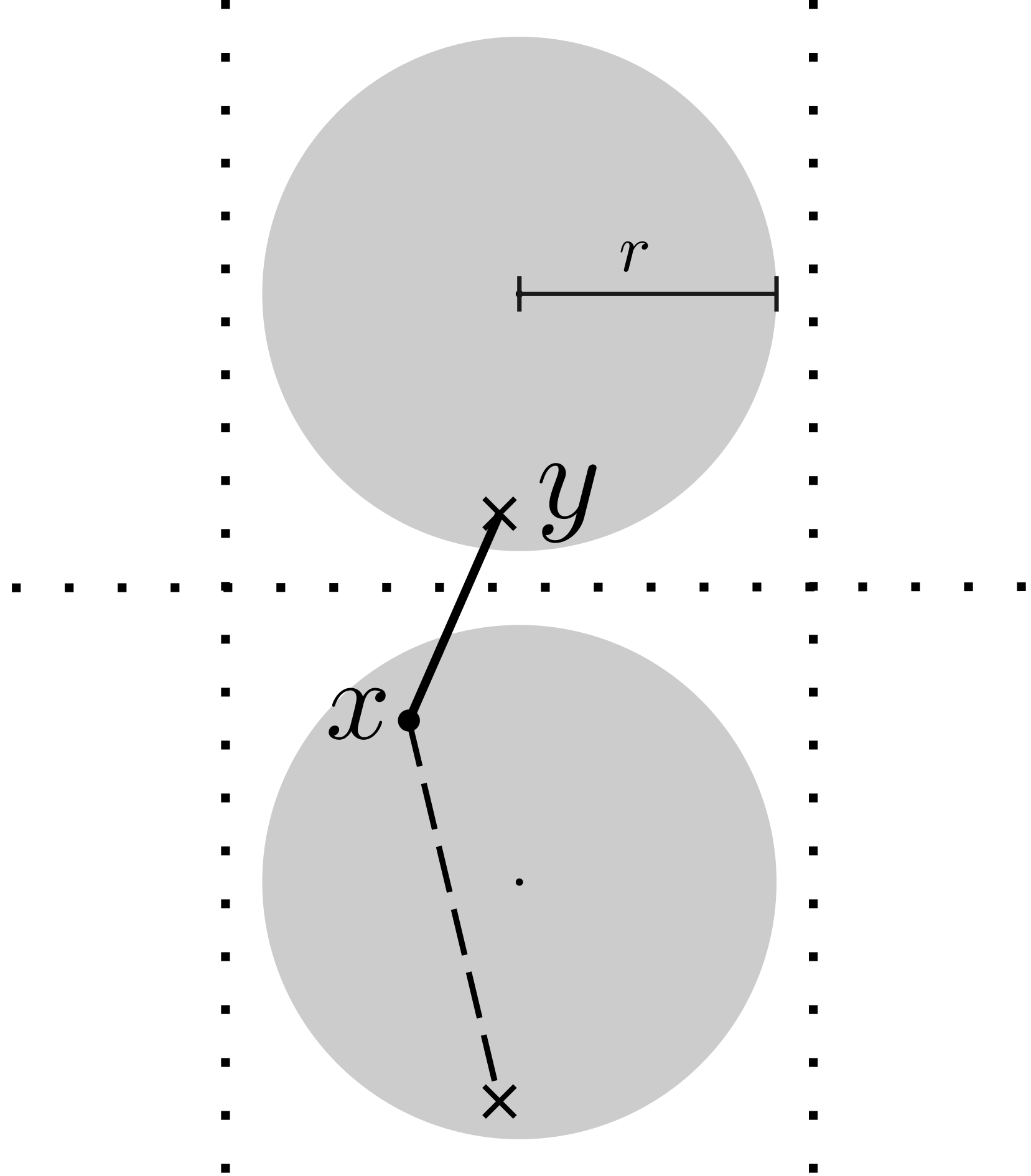}
		\caption{$r > \frac{1}{4}$}
		\label{fig:largeball}
	\end{subfigure}
	\hspace{4em}
	\begin{subfigure}[b]{0.3 \textwidth}
		\centering
		\includegraphics[width = \textwidth]{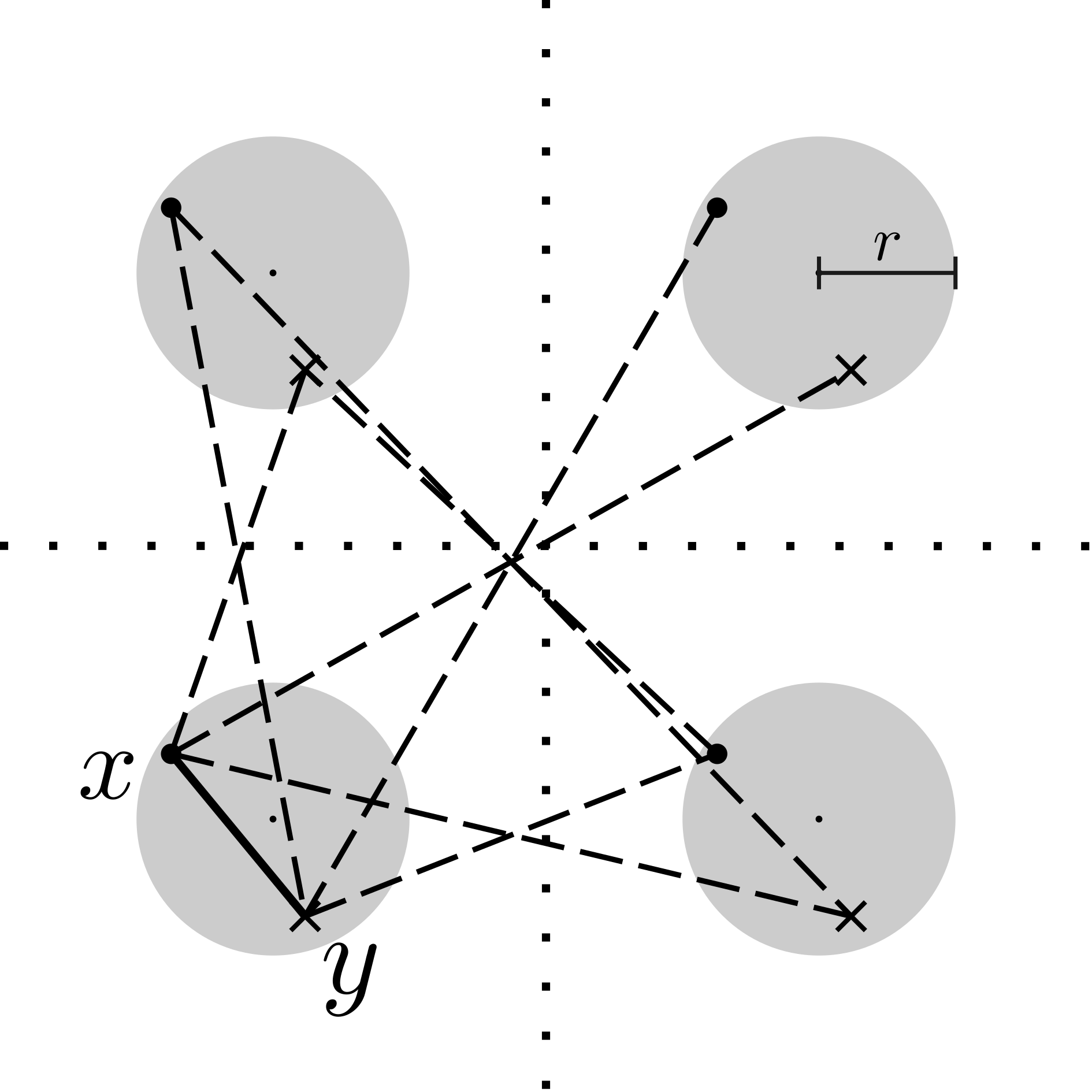}
		\caption{$r < \frac{1}{4}$}
		\label{fig:smallball}
	\end{subfigure}
	
	\caption{Straight line paths in $\TT^2$ between two points $x$ and $y$ contained in $\overline{B_r}(0)$. \subref{fig:largeball} When $r$ is large, the shortest path (solid line), which has length $|x-y|_{\TT^2}$, may pass outside of $\overline{B_r}(0)$. \subref{fig:smallball} When $r$ is small, the shortest path is contained in $\overline{B_r}(0)$. }
	\label{fig:toruspaths}
\end{figure}

\begin{lemma} \label{lem:torus_paths}
Let $x, y \in \overline{B_r}(0)$ for some $r < \frac{1}{4}$. Let $\tilde y \in \left [ - \frac{1}{2}, \frac{1}{2} \right )^d$ be an $\RR^d$-representative of $y$, and let $\tilde x \in \RR^d$ be a representative of $x$ such that $|x-y|_{\TT^d} = | \tilde x - \tilde y|$. Then $|(1-\lambda) \tilde y + \lambda \tilde x|_{\TT^d} \leq r$ for all $\lambda \in [0,1]$.
\end{lemma}

\begin{proof}
For all $z \in \RR^d$ such that $|z| < \frac{1}{2}$, $|z|_{\TT^d} = |z|$. Hence, if $x, y \in \overline{B_r}(0)$ for $r < \frac{1}{4}$, then $|\tilde x| \leq |\tilde y| + |\tilde x - \tilde y| \leq 2r < \frac{1}{2}$, which implies that $\tilde x \in \left [ - \frac{1}{2}, \frac{1}{2} \right )^d$ and $|x|_{\TT^d} = |\tilde x| \leq r$.
Then, since the ball $\overline{B_r}(0 ; \RR^d)$ in $\RR^d$ is convex, $|(1-\lambda) \tilde y + \lambda \tilde x | \leq r$ for all $\lambda \in [0,1]$. Since $r < \frac{1}{2}$, $|(1-\lambda) \tilde y + \lambda \tilde x |_{\TT^d} = |(1-\lambda) \tilde y + \lambda \tilde x |$, and the proof is complete.

\end{proof}

\begin{proof}[Proof of Lemma~\ref{lem:l}]
Let $x, y \in \TT^d$ with $|x-y|_{\TT^d} < r$.
Observe that both $h(x)$ and $h(y)$ are bounded from below by the supremum of $|\nabla g|$ over the intersection of the balls $\overline{B_r}(x)$ and $\overline{B_r}(y)$:
\be \label{est:hLB}
h(x), \;  h(y) \geq \sup_{u \in \overline{B_r}(y) \cap \overline{B_r}(x)} | \nabla g(u) | .
\ee

Since $\nabla g$ is a continuous function, there exists $z^\ast_y \in \overline{B_r}(y)$ such that $h(y) = | \nabla g(z^\ast_y) |$. If $z^\ast_y$ can be chosen within $\overline{B_r}(x) \cap \overline{B_r}(y)$, then 
\be \label{est:h_yx_shared}
h(y) = \sup_{u \in \overline{B_r}(y) \cap \overline{B_r}(x)} | \nabla g(u) | \leq h(x).
\ee

\begin{figure}[h]
  \centering
  \includegraphics[width=0.35\textwidth]{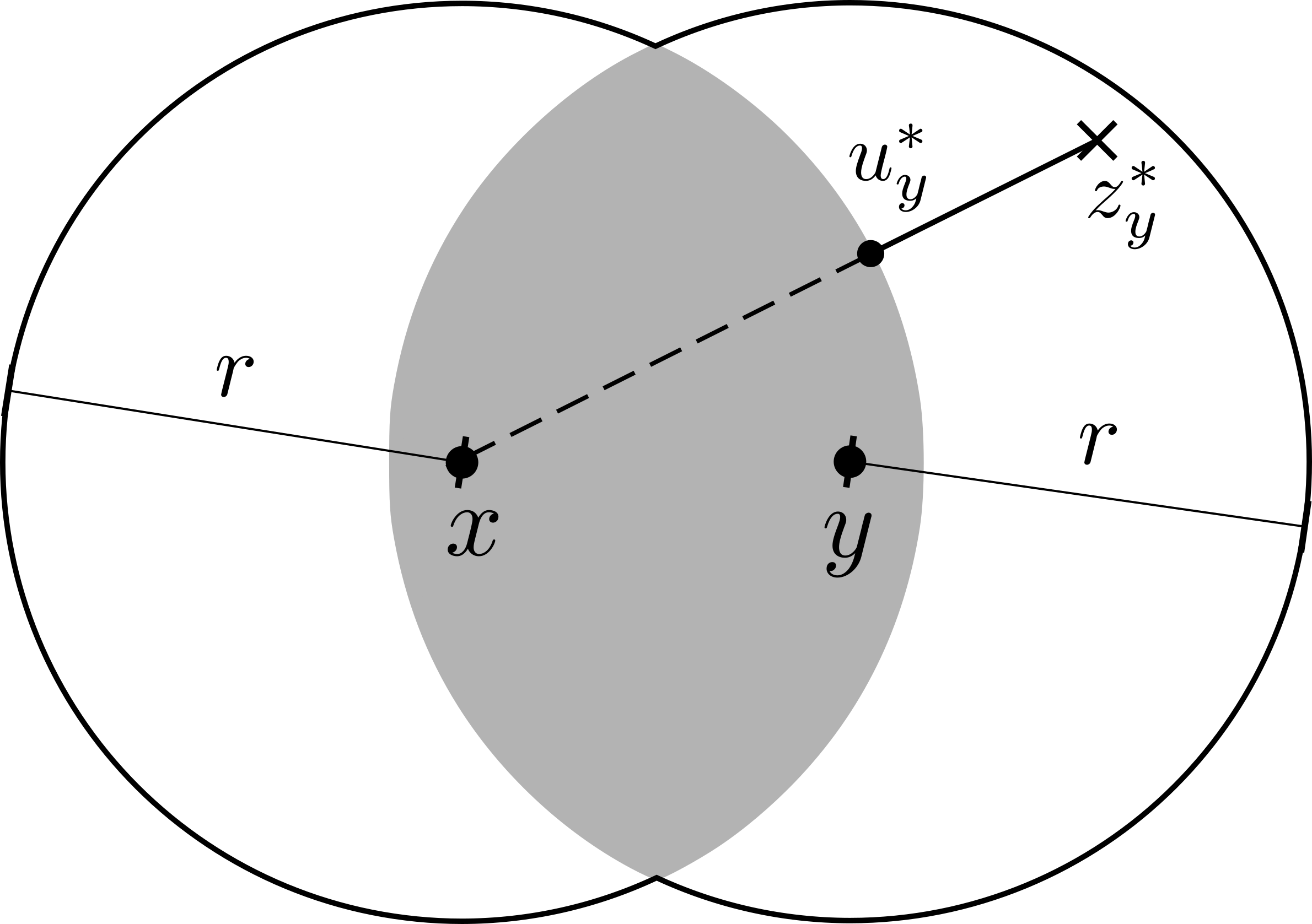}
  \caption{The union $\overline{B_r}(x) \cup \overline{B_r}(y)$. If an optimal $z^\ast_y$ lies in $\overline{B_r}(x) \cap \overline{B_r}(y)$ (shaded region), then $h(y) \leq h(x)$. Otherwise, we compare $|\nabla g(z^\ast_y)|$ with $|\nabla g(u^\ast_y)|$, by using a mean value theorem argument along the solid line segment connecting $u^\ast_y$ and $z^\ast_y$, whose length is no greater than $|x-y|_{\TT^d}$.} \label{fig:points}
\end{figure}

Otherwise, $z^\ast_y \in \overline{B_r}(y) \setminus \overline{B_r}(x)$.
In this case, we will compare $| \nabla g(z^\ast_y) |$ with $| \nabla g(u) |$ for some $u \in \overline{B_r}(x) \cap \overline{B_r}(y)$.
We wish to choose $u^\ast_y$ lying on the shortest straight line between $x$ and $z^\ast_y$ such that $| u^\ast_y - x |_{\TT^d} = r$ (see Figure~\ref{fig:points}).
This will ensure that $| z^\ast_y - u^\ast_y |_{\TT^d}$ can be bounded by $|x-y|_{\TT^d}$.

First, we define $u^\ast_y$. Since $x, z^\ast_y \in \overline{B_r}(y)$, we may identify $x, y, z^\ast_y$ with their representatives in $\R^d$ such that $|x-y| \leq r < \frac{1}{4}$ and $|z^\ast_y-y| \leq r < \frac{1}{4}$.
Then, by Lemma~\ref{lem:torus_paths} we have $| \lambda z^\ast_y + (1-\lambda) x - y |_{\TT^d} \leq r$ for all $\lambda \in [0,1]$---that is, this line is contained in $\overline{B_r}(y; \TT^d)$.
Since $| z^\ast_y - x |_{\TT^d} > r$, choosing $\lambda^\ast_y = \frac{r}{| z^\ast_y - x |_{\TT^d}} \in (0,1)$ gives a point $u^\ast_y : = x + \lambda^\ast_y ( z^\ast_y - x)$ satisfying $u^\ast_y \in \overline{B_r}(x) \cap \overline{B_r}(y)$ and $| u^\ast_y - x |_{\TT^d} = r$ as required.

Then, we estimate $| z^\ast_y - u^\ast_y |_{\TT^d}$. Using the definition of $u^\ast_y$,
\be 
| z^\ast_y - u^\ast_y |_{\TT^d} = (1 - \lambda^\ast_y ) | z^\ast_y - x |_{\TT^d} =  | z^\ast_y - x |_{\TT^d} - | u^\ast_y - x |_{\TT^d} .
\ee
By the triangle inequality, since $ | z^\ast_y - y |_{\TT^d} \leq r$ and $| u^\ast_y - x |_{\TT^d} = r$,
\be \label{est:dist}
| z^\ast_y - u^\ast_y |_{\TT^d} \leq | x - y |_{\TT^d} + | z^\ast_y - y |_{\TT^d} - | u^\ast_y - x |_{\TT^d} \leq | x - y |_{\TT^d}.
\ee

Now we compare $| \nabla g(z^\ast_y) |$ with $| \nabla g(u^\ast_y) |$.
By the mean value theorem,
\begin{align} \label{est:MVT}
| \nabla g(z^\ast_y) - \nabla g(u^\ast_y)| &\leq | z^\ast_y - u^\ast_y |_{\TT^d} \sup_{\lambda \in [0,1]} |\nabla^2 g(\lambda z^\ast_y + (1-\lambda) u^\ast_y)| \\
& \leq | z^\ast_y - u^\ast_y |_{\TT^d} \sup_{z \in \overline{B_r}(y)} |\nabla^2 g(z)| , \label{SwitchBreak}
\end{align}
since for all $\lambda \in [0,1]$,
$\lambda z^\ast_y + (1-\lambda) u^\ast_y = \lambda ' z^\ast_y + (1-\lambda ') x$ for some $\lambda ' \in [0,1]$ so that $\lambda z^\ast_y + (1-\lambda) u^\ast_y \in \overline{B_r}(y)$.
Hence, by \eqref{est:dist} and the definition \eqref{def:l} of $l$,
\be
| \nabla g(z^\ast_y) - \nabla g(u^\ast_y) | \leq l(y) | x - y |_{\TT^d}.
\ee
Then, by \eqref{est:hLB}, since $u^\ast_y \in \overline{B_r}(x) \cap \overline{B_r}(y)$,
\begin{align}
h(y) = | \nabla g(z^\ast_y) | & \leq  | \nabla g(u^\ast_y) | + l(y) | x - y |_{\TT^d} \\
& \leq \sup_{u \in \overline{B_r}(y) \cap \overline{B_r}(x)} | \nabla g(u) | + l(y) | x - y |_{\TT^d} \\
& \leq h(x) + l(y) | x - y |_{\TT^d} . \label{est:h_yx_sep}
\end{align}
From \eqref{est:h_yx_shared} and \eqref{est:h_yx_sep}, we deduce that in either case
\be \label{est:h_yx}
h(y) \leq h(x) + l(y) | x - y |_{\TT^d} .
\ee

We repeat the argument, exchanging the roles of $x$ and $y$, as far as \eqref{SwitchBreak}, to obtain that either $h(x) \leq h(y)$, or $z^\ast_x \in \overline{B_r}(x) \setminus \overline{B_r}(y)$ and
\be \label{est:MVTx}
| \nabla g(z^\ast_x) - \nabla g(u^\ast_x)| \leq | x - y |_{\TT^d} \sup_{z \in \overline{B_r}(x)} |\nabla^2 g(z)| ,
\ee
where $u^\ast_x : = y + r \frac{z^\ast_x - y}{| z^\ast_x - y |_{\TT^d}} \in  \overline{B_r}(x) \cap \overline{B_r}(y)$.
Since $| x - y |_{\TT^d} < r$, we have $\overline{B_r}(x) \subset \overline{B_{2r}}(y)$. Hence
\be \label{est:x_eta}
| \nabla g(z^\ast_x) - \nabla g(u^\ast_x)| \leq l(y) | x - y |_{\TT^d}  .
\ee
Since $u^\ast_x \in \overline{B_r}(x) \cap \overline{B_r}(y)$, we deduce from \eqref{est:hLB} that
\be \label{est:h_xy}
h(x) \leq h(y) +  l(y) | x - y |_{\TT^d} .
\ee

Taking estimates \eqref{est:h_yx}-\eqref{est:h_xy} completes the proof that $l$ is an $r$-local Lipschitz estimate for $h$.

\end{proof}

For the derivation of the ionic Vlasov-Poisson system, the cases of interest will be $g = \chi_r$ and $g = K \ast \chi_r$. We now apply the above results in order to identify suitable functions $h$ and $l$ for these cases.

\begin{definition} \label{def:moduli}
 Let $r \in (0, \frac{1}{4})$. Let $\chi_r$ be defined as in Definition~\ref{def:chi1} for some non-negative, radially symmetric $\chi \in C^\infty(\R^d ; [0 +\infty) )$ with compact support contained in $\overline{B_1}(0)$ and unit mass $\| \chi \|_{L^1} = 1$. Then
\begin{enumerate}[(i)]
\item $K_r : \TT^d \to \R^d$ denotes the function $K_r : = K \ast \chi_r$.

$L_r : \TT^d \to [0, + \infty)$ denotes the function
\be
L_r(y) : = \sup_{z \in B_r(y)} | \nabla K_r( z) | \qquad \text{ for all } y \in \TT^d .
\ee
Then $L_r$ is an $r$-local Lipschitz modulus for $K_r$.

$Q_r : \TT^d \to [0, + \infty)$ denotes the function
\be
Q_r(y) : = \sup_{z \in B_{2r}(y)} | \nabla^2 K_r( z) | \qquad \text{ for all } y \in \TT^d .
\ee
Then $Q_r$ is an $r$-local Lipschitz modulus for $L_r$.
\item
$\psi_r : \TT^d \to [0, + \infty)$ denotes the function
\be
\psi_r(y) : = \sup_{z \in B_r(y)} | \nabla \chi_r( z) | \qquad \text{ for all } y \in \TT^d .
\ee
Then $\psi_r$ is an $r$-local Lipschitz modulus for $\chi_r$. 

$\eta_r : \TT^d \to [0, + \infty)$ denotes the function
\be \label{def:eta}
\eta_r(y) : = \sup_{z \in B_{2r}(y)} |\nabla^2 \chi_r( z)| \qquad \text{ for all } y \in \TT^d .
\ee
Then $\eta_r$ is an $r$-local Lipschitz modulus for $\psi_r$, by Lemma~\ref{lem:l}.
\end{enumerate}

\end{definition}

We collect integrability estimates for these functions. These will be used later to bound the rates of convergence in probability.

\begin{lemma} \label{lem:Fn_bounds}
Let $\chi_r$, $\psi_r$, $\eta_r$, $K_r$, $L_r$ and $Q_r$ be defined as in Definition~\ref{def:moduli} for a fixed $\chi$ and any $r \in (0, \frac{1}{4})$.

\begin{enumerate}[(i)]

\item \label{item:chi_ests} The functions $\chi_r$, $\psi_r$, and $\eta_r$ satisfy the pointwise bounds
\be \label{est:chi_ptwise}
\chi_r(y) \lesssim_\chi r^{-d} \one_{B_{r}(0)}(y) , \qquad \psi_r(y) \lesssim_\chi r^{-d-1} \one_{B_{2r}(0)}(y) , \qquad \eta_r(y) \lesssim_\chi r^{-d-2} \one_{B_{3r}(0)}(y) ,
\ee
the $L^1$ bounds
\be \label{est:chi_L1}
\| \chi_r \|_{L^1} \leq 1 , \qquad \| \psi_r \|_{L^1} \lesssim_{d, \chi}  r^{-1}  , \qquad \| \eta_r \|_{L^1} \lesssim_{d, \chi}  r^{-2} ,
\ee
and the $L^2$ bounds
\be \label{est:chi_L2}
\| \chi_r \|_{L^2} \lesssim_{d, \chi} r^{-d/2} , \qquad \| \psi_r \|_{L^2} \lesssim_{d, \chi}  r^{-d/2 -1}  , \qquad \| \eta_r \|_{L^2} \lesssim_{d, \chi}  r^{-d/2 -2} .
\ee

\item \label{item:K_ests} The functions $K_r$, $L_r$, and $Q_r$ satisfy the following pointwise bounds: for $y \in \TT^d$,
\be \label{est:K_ptwise}
| K_r(y) | \lesssim_{d, \chi} ( |y|_{\TT^d} \vee 2r)^{1-d} , \qquad L_r(y) \lesssim_{d, \chi} ( |y|_{\TT^d} \vee 3r)^{-d} , \qquad Q_r(y)  \lesssim_{d, \chi} ( |y|_{\TT^d} \vee 4r)^{-d - 1} ,
\ee
the $L^1$ bounds
\be \label{est:K_L1}
\| K_r \|_{L^1(\TT^d)} \lesssim_{d, \chi} 1 , \qquad \| L_r \|_{L^1(\TT^d)} \lesssim_{d, \chi} |\log r| , \qquad \| Q_r \|_{L^1(\TT^d)} \lesssim_{d, \chi} r^{-1} ,
\ee
and the $L^2$ bounds
\be \label{est:K_L2}
\| K_r \|_{L^2(\TT^d)} \lesssim_{d, \chi} r^{1-d/2} , \qquad \| L_r \|_{L^2(\TT^d)} \lesssim_{d, \chi} r^{-d/2} , \qquad \| Q_r \|_{L^2(\TT^d)} \lesssim_{d, \chi} r^{-1 -d/2 } .
\ee
\end{enumerate}

\end{lemma}

\begin{proof}
\eqref{item:chi_ests} Since $\chi_r(y) = r^{-d} \chi (y/r)$, we have
\be
\nabla \chi_r(y) = r^{-d - 1} \nabla \chi (y/r), \qquad \nabla^2 \chi_r(y) = r^{-d - 2} \nabla^2 \chi (y/r) .
\ee
The estimates \eqref{est:chi_ptwise} then follow from the fact that $\chi (y/r)$ is non-zero only if $|y|_{\TT^d} < r$. The $L^1$ bound \eqref{est:chi_L1} on $\chi_r$ holds since $\| \chi_r \|_{L^1} = \| \chi \|_{L^1} = 1$.
We deduce the remaining $L^1$ bounds \eqref{est:chi_L1} and $L^2$ bounds \eqref{est:chi_L2} by integrating the estimates \eqref{est:chi_ptwise} over all $y \in \TT^d$.

\noindent \eqref{item:K_ests} We note that, by \eqref{def:K1}, the Coulomb kernel $K$ is an $L^1(\TT^d)$ function. Then, since $K_r = \chi_r \ast K$, by Young's inequality for convolutions and \eqref{est:chi_ptwise} we have
\be \label{est:Kr_smally}
\| K_r \|_{L^\infty(\TT^d)} \leq \| K \|_{L^1(\TT^d)} \| \chi_r \|_{L^\infty(\TT^d)} \lesssim_{d, \chi} r^{-d} .
\ee
We will use \eqref{est:Kr_smally} in the case where $|y|_{\TT^d} \leq 2r$.

Otherwise, we write out the convolution to obtain the expression
\be
K_r(y) = \int_{\TT^d} K(y - y') \chi_r(y') \dd y ' .
\ee
Since $\chi_r(y') = 0$ when $|y ' |_{\TT^d} > r$, then
\be
| K_r(y) | \lesssim_d \int_{\TT^d}  | y - y' |^{1-d}_{\TT^d} \one_{\{ |y ' |_{\TT^d} \leq r \}} \, \chi_r(y') \dd y ' .
\ee
If $|y|_{\TT^d} > 2r$ and $|y ' |_{\TT^d} \leq r$, then $| y - y' |_{\TT^d} \geq |y|_{\TT^d} - |y ' |_{\TT^d} > \frac{1}{2} |y|_{\TT^d}$. Hence
\be \label{est:Kr_bigy}
| K_r(y) | \lesssim_d |y|_{\TT^d}^{1-d} \int_{\TT^d} \chi_r(y') \dd y ' \lesssim_d |y|_{\TT^d}^{1-d} .
\ee
Combining \eqref{est:Kr_smally} and \eqref{est:Kr_bigy} results in the estimate \eqref{est:K_ptwise} for $K_r$.

For $L_r$, we note first that we may write $\nabla K_r$ as either $K \ast \nabla \chi_r$ or $\nabla K \ast \chi_r$. The first case gives us the global estimate
\be
\| \nabla K_r \|_{L^\infty(\TT^d)} \lesssim_d \| K \|_{L^1(\TT^d)} \| \nabla \chi_r \|_{L^\infty(\TT^d)} \lesssim_{d, \chi} r^{-d - 1} . 
\ee
We will use the second case to estimate $\nabla K_r(z)$ where $z \in B_r(y)$ and $|y|_{\TT^d} > 3r$. In particular this implies that $|z|_{\TT^d} > \frac{2}{3} |y|_{\TT^d} > 2r$.
Then
\be
| \nabla K_r(z) | = \left | \int_{\TT^d} \nabla K(z - z') \chi_r(z') \dd z ' \right | \lesssim_d \left | \int_{\TT^d} | z - z'|_{\TT^d}^{-d} \one_{\{ |z ' |_{\TT^d} \leq r \}} \,  \chi_r(z') \dd z ' \right |.
\ee
We apply the reverse triangle inequality to obtain the lower bound $| z - z'|_{\TT^d} \geq | z|_{\TT^d}  - | z'|_{\TT^d} > \frac{2}{3} |y|_{\TT^d} - r > \frac{1}{3} |y|_{\TT^d}$.
Hence 
\be
\sup_{z \in B_r(y)} | \nabla K_r(z) | \lesssim_d  | y |_{\TT^d}^{-d} , \qquad |y|_{\TT^d} > 3r,
\ee
and we deduce the pointwise estimate \eqref{est:K_ptwise} for $L_r$. The estimate for $Q_r$ follows in a similar way from the bound $\nabla^2 K(y) \lesssim_d |y|_{\TT^d}^{-d-1}$.

Finally, the $L^1$ bounds \eqref{est:K_L1} and $L^2$ bounds \eqref{est:K_L2} follow from integrating the pointwise bounds \eqref{est:K_ptwise} over $y\in \TT^d$; see e.g. \cite[Proposition 7.2]{LazaroviciPickl} for similar estimates.

\end{proof}

\subsection{Estimates Between Empirical Measures}

As a first application of local Lipschitz bounds, we obtain $L^p$-type stability estimates for the convolution of a locally Lipschitz function with an empirical measure. If two particle systems are within distance $r$ of each other, then the difference between the convolutions with their empirical measures can be bounded using an estimate with a `weak-strong' form, in which the bound depends on the convolution of the Lipschitz modulus with only one of the empirical measures.

\begin{lemma} \label{lem:gApproxEmpirical}
Let ${\bf X} = ( X_i )_{i=1}^N, {\bf Y} = ( Y_i )_{i=1}^N \in (\TT^d)^N$ be two $N$-tuples of points in $\TT^d$.
Let $g : \TT^d \to \R$ be a Lipschitz continuous function with $r$-local Lipschitz modulus $h$.
Assume that $| {\bf X} - {\bf Y} |_{\infty} < r$. Then, for any $p \in [1, + \infty]$,
\be \label{eq:g_rhoXY_stab}
\| g\ast \mu_{ {\bf X}} - g \ast \mu_{{\bf Y}} \|_{L^p} \leq \| h \ast \mu_{{\bf Y}} \|_{L^p} |{\bf X} -{\bf Y}|_{\infty} .
\ee

\end{lemma}
\begin{proof}
By the definitions of the empirical measures $\mu_{{\bf X}}$ and $\mu_{{\bf Y}}$, and the triangle inequality, 
\begin{align}
| g \ast \mu_{{\bf X}} (y) - g \ast \mu_{{\bf Y}} (y) | & = \left | \frac{1}{N} \sum_{i=1}^N \Big [ g (y- X_i ) - g (y- Y_i ) \Big ] \right |  \\
& \leq \frac{1}{N} \sum_{i=1}^N \left | g (y- X_i ) - g (y- Y_i ) \right | .
\end{align}
Since $|X_i - Y_i| < r$ for all $i = 1, \ldots, N$, by the property \eqref{eq:ModOfCont} of $g$ and $h$ (choosing, for each $i$, to evaluate $h$ at $y - Y_i$),
we obtain the upper bound
\begin{align}
\frac{1}{N} \sum_{i=1}^N \left | g (y- X_i ) - g (y- Y_i ) \right | & \leq  \frac{1}{N} \sum_{i=1}^N  h(y- Y_i) \, |X_i - Y_i |  \\
& \leq \left ( \frac{1}{N} \sum_{i=1}^N h(y- Y_i) \right )   |{\bf X} - {\bf Y}|_{\infty} \, \\
& \leq h \ast \mu_{{\bf Y}}(y) \,  |{\bf X} - {\bf Y}|_{\infty} .
\end{align}
The bound \eqref{eq:g_rhoXY_stab} then follows from taking the $L^p$ norm of both sides with respect to $y \in \TT^d$.
\end{proof}

\subsection{Law of Large Numbers}

We now turn to the law of large numbers result for convolutions with empirical measures $g \ast \mu_{\bf Y}$, where ${\bf Y} = \{ Y_i\}_{i=1}^N$ is an $N$-tuple of independent and identically distributed random variables.
The goal is to obtain rates of convergence depending on the local Lipschitz modulus. 
Our first result reduces the problem of convergence in $L^\infty$ to the problem of convergence on a finite mesh of points.

\begin{lemma} \label{lem:WeakStrongLinfty}
Let $\nu_1, \nu_2$ be measures on $\TT^d$. 
Consider a collection of $M$ points $\{ y_m \}_{m=1}^M \subset \TT^d$ such that $\TT^d \subset \bigcup_{m=1}^M B_r(y_m)$.
Let $g$ be a Lipschitz function with $r$-local Lipschitz modulus $h$. 
Then
\be
\| g \ast \nu_1 - g \ast \nu_2 \|_{L^\infty} 
 \leq 2r \| h \ast \nu_1 \|_{L^\infty} + \sup_{m} |g \ast \nu_1(y_m) - g \ast \nu_2(y_m)| 
+ r \sup_m  |h \ast \nu_1(y_m) -  h \ast \nu_2(y_m) | .
\ee

\end{lemma}
\begin{remark}
We emphasise as before that the $L^\infty$ norm is evaluated only on the convolutions with $\nu_1$.
\end{remark}

\begin{proof}
By assumption, the sets $\{ B_r(y_m) \}_{m=1}^M$ form an open cover of $\TT^d$.
Let $\{ \zeta_m \}_{m=1}^M$ denote a continuous partition of unity subordinate to this cover.
That is, each continuous function $\zeta_m$ takes non-negative values and is supported in the set $B_r(y_m)$, while
\be \label{zeta_partof1}
\sum_{m=1}^M \zeta_m(y) = 1  \quad \text{for all} \; y \in \TT^d.
\ee

For each $i=1,2$, we will approximate $g \ast \nu_i$ by interpolating its values at the points $\{ y_m \}_{m=1}^M$, using the functions $\{\zeta_m \}_{m=1}^M$, namely:
\be
g \ast \nu_i \simeq  \sum_{m=1}^M  g \ast \nu_i (y_m) \, \zeta_m .
\ee
To estimate the error in this approximation, we use \eqref{zeta_partof1} to write, for all $y \in \TT^d$,
\begin{align}
\left | g \ast \nu_i (y) - \sum_{m=1}^M g \ast \nu_i(y_m) \, \zeta_m(y)   \right | & =  \left |  g \ast \nu_i (y) \sum_{m=1}^M \zeta_m(y)  - \sum_{m=1}^Mg \ast \nu_i(y_m) \, \zeta_m(y)   \right | \\
& =  \left |  \sum_{m=1}^M \left ( g \ast \nu_i (y) - g \ast \nu(y_m) \right ) \zeta_m(y)  \right | \\
& \leq \sum_{m=1}^M  \left |  g \ast \nu_i (y) - g \ast \nu_i(y_m) \right | \zeta_m(y) .
\end{align}
By expanding the convolution, we find that
\be
 \left |  g \ast \nu_i (y) - g \ast \nu_i(y_m) \right | \leq \int_{\TT^d}  |g(y - y') - g(y_m - y') | \dd \nu_i(y') .
\ee

For all $y \in \TT^d$ such that $\zeta_m(y)>0$, and any $y ' \in \TT^d$, we have $| ( y - y') - (y_m - y')| = |y - y_m| < r$. Hence we may apply the property \eqref{eq:ModOfCont} of $g$ and $h$. Choosing to evaluate $h$ at the point $y - y'$ leads to the following estimate, valid for each $m = 1, \ldots M$:
\begin{align}
 \left |  g \ast \nu_i (y) - g \ast \nu_i(y_m) \right | \zeta_m(y) & \leq \zeta_m(y) \int_{\TT^d}  h (y - y') \, |y - y_m| \dd \nu_i(y') \\
 & \leq |y - y_m| \zeta_m(y) \,h \ast \nu_i(y) \\
 & \leq r \zeta_m(y) \,h \ast \nu_i(y) .
\end{align}
Hence, by summing these inequalities and using \eqref{zeta_partof1}, we obtain that for all $y \in \TT^d$,
\be \label{est:conv-Lip}
\left | g \ast \nu_i (y) - \sum_{m=1}^M g \ast \nu_i(y_m) \, \zeta_m(y)   \right | \leq r \,h \ast \nu_i(y)  \sum_{m=1}^M \zeta_m(y) = r h \ast \nu_i(y) .
\ee

If we choose instead to evaluate $h$ at the point $y_m - y'$, we obtain the following alternative estimate for $m = 1, \ldots M$:
\begin{align}
 \left |  g \ast \nu_i (y) - g \ast \nu_i(y_m) \right | \zeta_m(y) & \leq \zeta_m(y) \int_{\TT^d}  h (y_m - y') \, |y - y_m| \dd \nu_i(y') \\
 & \leq r \zeta_m(y) \,h \ast \nu_i(y_m).
\end{align} 
Summing these contributions, we obtain for all $y \in\TT^d$
\be \label{est:conv-Lip-approx}
\left | g \ast \nu_i (y) - \sum_{m=1}^M g \ast \nu_i(y_m) \, \zeta_m(y) \right | \leq r \, \sum_{m=1}^M \zeta_m(y)\, h \ast \nu_i(y_m) .
\ee

Next, we will apply the estimates \eqref{est:conv-Lip} and \eqref{est:conv-Lip-approx} to the estimation of 
\begin{multline} \label{total-triangle}
\left | g \ast \nu_1 - g \ast \nu_2 \right |  \leq \left | \sum_{m=1}^M ( g \ast \nu_1(y_m) - g \ast \nu_2(y_m) ) \, \zeta_m \right | \\
 +  \left | g \ast \nu_1 - \sum_{m=1}^M g \ast \nu_1(y_m) \, \zeta_m   \right | + \left | g \ast \nu_2 - \sum_{m=1}^M g \ast \nu_2 (y_m) \, \zeta_m   \right | .
\end{multline}
We bound the first term by taking supremum over the quantities $|g \ast \nu_1(y_m) - g \ast \nu_2(y_m)|$:
\begin{align}
\left | \sum_{m=1}^M ( g \ast \nu_1(y_m) - g \ast \nu_2(y_m) ) \, \zeta_m(y) \right | & \leq \left ( \sup_{m} |g \ast \nu_1(y_m) - g \ast \nu_2(y_m)| \right ) \sum_{m' = 1}^M \zeta_{m'} \\
& \leq \sup_{m} |g \ast \nu_1(y_m) - g \ast \nu_2(y_m)|, \label{approx-dev}
\end{align}
where the last inequality is a consequence of \eqref{zeta_partof1}.
We use inequality \eqref{est:conv-Lip} to estimate the deviation of $\nu_1$, obtaining
\be \label{gnu1-dev}
\left | g \ast \nu_1 - \sum_{m=1}^M g \ast \nu_1(y_m) \, \zeta_m   \right | \leq r \, h \ast \nu_1 .
\ee
For $\nu_2$, we instead apply \eqref{est:conv-Lip-approx} to obtain
\be \label{nu2-dev-hnu2}
\left | g \ast \nu_2 - \sum_{m=1}^M g \ast \nu_2 (y_m) \, \zeta_m   \right | \leq r \, \sum_{m=1}^M  h \ast \nu_2(y_m)\,  \zeta_m.
\ee
It remains to estimate the right hand side of \eqref{nu2-dev-hnu2} using quantities involving $\nu_1$.

First, applying the triangle inequality gives the bound
\be
\sum_{m=1}^M  h \ast \nu_2(y_m)\,  \zeta_m \leq \sum_{m=1}^M  h \ast \nu_1(y_m)\,  \zeta_m + \sum_{m=1}^M  |h \ast \nu_1(y_m) -  h \ast \nu_2(y_m) | \,  \zeta_m .
\ee
The second term may be estimated by taking supremum over $|h \ast \nu_1(y_m) -  h \ast \nu_2(y_m) |$ and applying \eqref{zeta_partof1} to obtain
\be \label{nu2happrox-tri}
\sum_{m=1}^M  h \ast \nu_2(y_m)\,  \zeta_m \leq \sum_{m=1}^M  h \ast \nu_1(y_m)\,  \zeta_m + \sup_m  |h \ast \nu_1(y_m) -  h \ast \nu_2(y_m) | .
\ee
We then deduce that
\begin{align} 
\left | \sum_{m=1}^M  h \ast \nu_2(y_m)\,  \zeta_m \right | & \leq \| h \ast \nu_1 \|_{L^\infty} \sum_{m=1}^M  \zeta_m + \sup_m  |h \ast \nu_1(y_m) -  h \ast \nu_2(y_m) | \\
& \leq \| h \ast \nu_1 \|_{L^\infty} + \sup_m  |h \ast \nu_1(y_m) -  h \ast \nu_2(y_m) |  . \label{nu2happrox-tri2}
\end{align}

We combine \eqref{nu2-dev-hnu2} with \eqref{nu2happrox-tri2} to obtain
\be \label{gnu2-dev-infty}
\left | g \ast \nu_2 - \sum_{m=1}^M g \ast \nu_2 (y_m) \, \zeta_m   \right | \leq r \| h \ast \nu_1 \|_{L^\infty} + r \sup_m  |h \ast \nu_1(y_m) -  h \ast \nu_2(y_m) | .
\ee
Substituting
\eqref{approx-dev} \eqref{gnu1-dev} and \eqref{gnu2-dev-infty} into \eqref{total-triangle} gives
\be
\left | g \ast \nu_1 - g \ast \nu_2 \right | \leq  r \, h \ast \nu_1 + r \| h \ast \nu_1 \|_{L^\infty} + \sup_{m} |g \ast \nu_1(y_m) - g \ast \nu_2(y_m)| + r \sup_m  |h \ast \nu_1(y_m) -  h \ast \nu_2(y_m) | .
\ee
Taking the supremum on both sides completes the proof.

\end{proof}

Next, we consider the law of large numbers setting, taking the measure $\nu_2$ to be an empirical measure induced by a collection of independent, identically distributed random variables. We prove the following rate of convergence in probability as the number of samples tends to infinity.

\begin{prop} \label{prop:LLN}
Let $\rho \in L^\infty_+(\TT^d)$ be a bounded probability density on $\TT^d$.
Let ${\bf Y}: \Omega \to (\TT^d)^N$ be a random $N$-tuple on $(\Omega, \cF, \PP)$ such that ${\bf Y} \sim \rho^{\otimes N}$ under $\PP$.

For $r > 0$, let $g : \TT^d \to \R$ be a Lipschitz continuous function with $r$-local Lipschitz modulus $h \in L^\infty(\TT^d)$.
For any $\gamma > 0$ and  $\delta \in (0,1)$ let $A_g = A_g(\gamma, \delta) \in \cF$ denote the event 
\be
A_g : = \Big  \{ \| g \ast \mu_{{\bf Y}} - g \ast \rho \|_{L^\infty} \leq 2 r \| \rho \|_{L^\infty}  \left (  \delta \| h  \|_{L^1} +  \gamma \right ) \Big \} .
\ee
Then
\be
\PP(A_g^c) \lesssim (r \delta)^{-d} \left ( \exp{\left (- N \| \rho \|_{L^\infty} \frac{r^2 \gamma^2}{2 (\| g  \|_{L^2}^2 + r \gamma \| g \|_{L^\infty} /3)} \right )} +  \exp{\left (- N \| \rho \|_{L^\infty} \frac{ \gamma^2}{2( \| h \|_{L^2}^2 + \gamma \| h \|_{L^\infty}/3 )} \right )} \right ) .
\ee
\end{prop}

The proof of this proposition uses the Bernstein inequality \cite{Bernstein}, which we state below; see e.g. \cite{BoucheronLugosiMassart} for proofs and other versions.

\begin{thm} \label{thm:Bernstein}
Let $U_1, \ldots, U_N : \Omega \to \R$ be independent identically distributed real-valued random variables that are bounded, i.e. there exists $B > 0$ such that $|U_i| \leq B$. Then, for all $\xi> 0$,
\be
\PP \left ( \left | \frac{1}{N} \sum_{j=1}^N U_j - \EE[U_1] \right | > \xi \right ) \leq 2 \exp{\left (- N \frac{\xi^2}{2(\Var U_1 + B \xi /3 )} \right )} .
\ee

\end{thm}

\begin{proof}[Proof of Proposition~\ref{prop:LLN}]

Fix a collection of points $\{ y_m \}_{m=1}^M \subset \TT^d$ such that the sets $\{ B_{r\delta}(y_m) \}_{m=1}^M$ form a cover of $\TT^d$ (i.e., $\TT^d \subset \bigcup_{m=1}^M B_{r \delta}(y_m)$), and $M \lesssim (r \delta)^{-d}$.
Since $h$ is an $r$-local Lipschitz modulus for $g$, it is also an $(r \delta )$-local Lipschitz modulus. Hence, by Lemma~\ref{lem:WeakStrongLinfty},
\be \label{est:g_dev_1}
\| g \ast \mu_{{\bf Y}} - g \ast \rho \|_{L^\infty} \leq 2 r \delta \| h \ast \rho \|_{L^\infty} +  r \delta \sup_m | h \ast \rho (y_m) - h \ast \mu_{{\bf Y}} (y_m)| 
+ \sup_m | g \ast \rho (y_m) - g \ast \mu_{{\bf Y}} (y_m)| .
\ee
Using Young's inequality for convolutions, we can bound $ \| h \ast \rho \|_{L^\infty} \leq  \| h \|_{L^1} \| \rho \|_{L^\infty}$.
Thus $\| g \ast \mu_{{\bf Y}} - g \ast \rho \|_{L^\infty} \leq 2 r \| \rho \|_{L^\infty} \left (  \delta \| h \|_{L^1} +  \gamma \right )$ holds if 
\be
 \sup_m | g \ast \rho (y_m) - g \ast \mu_{{\bf Y}} (y_m)| \leq r \gamma \| \rho \|_{L^\infty} \quad \text{and} \quad \sup_m | h \ast \rho (y_m) - h \ast \mu_{{\bf Y}} (y_m)| \leq \gamma \| \rho \|_{L^\infty} .
\ee
In other words,
\be
A_g^c \subset \bigcup_{m=1}^M \{ \left | g \ast \mu_{{\bf Y}} (y_m) - g \ast \rho (y_m) \right | > r \gamma \| \rho \|_{L^\infty} \} \cup \{ \left | h \ast \mu_{{\bf Y}} (y_m) - h \ast \rho (y_m) \right | > \gamma \| \rho \|_{L^\infty} \} .
\ee
We may therefore estimate
\be
\PP(A^c_g) 
 \leq \sum_{m=1}^M \PP ( \left | g \ast \mu_{{\bf Y}} (y_m) - g \ast \rho (y_m) \right | > r \gamma \| \rho \|_{L^\infty} ) + \PP(\left | h \ast \mu_{{\bf Y}} (y_m) - h \ast \rho (y_m) \right | > \gamma \| \rho \|_{L^\infty} ) . \label{DeMorgan}
\ee
We now bound each of the probabilities in the sum.

Observe that, by expanding the convolution and using the definition of the empirical measure $\mu_{{\bf Y}}$, we may rewrite
\be
g \ast \rho (y_m) - g \ast \mu_{{\bf Y}} (y_m) = \EE[g(y_m - Y_1)] - \frac{1}{N} \sum_{i=1}^N g(y_m - Y_i) .
\ee
The random variables $g(y_m - Y_i)$ are independent and bounded by
\be
|g(y_m - Y_i)| \leq \| g \|_{L^\infty} ,
\ee
with
variance
\be
\Var g(y_m - Y_1) = g^2 \ast \rho(y_m) - (g \ast \rho (y_m))^2 \leq \| g^2 \ast \rho \|_{L^\infty} \leq \| g \|_{L^2}^2 \| \rho \|_{L^\infty}.
\ee
We may therefore apply Theorem~\ref{thm:Bernstein}, choosing $\xi = r \gamma \| \rho \|_{L^\infty}$. We obtain, for each $m = 1, \ldots, M$,
\be \label{est:g-prob}
\PP \left ( \left | g \ast \mu_{{\bf Y}} (y_m) - g \ast \rho (y_m) \right | > r \gamma \| \rho \|_{L^\infty} \right ) \lesssim \exp{\left (- N \| \rho \|_{L^\infty}  \frac{r^2 \gamma^2}{2( \| g \|_{L^2}^2 + r \gamma \| g \|_{L^\infty}/3 ) } \right )} .
\ee
Similarly
\be \label{est:h-prob}
\PP \left ( \left | h \ast \mu_{{\bf Y}} (y_m) - h \ast \rho (y_m) \right | > \gamma \| \rho \|_{L^\infty} \right ) \lesssim \exp{\left (- N \| \rho \|_{L^\infty} \frac{ \gamma^2}{2(\| h \|_{L^2}^2 + \gamma \| h \|_{L^\infty}/3) } \right )} .
\ee

Substituting inequalities \eqref{est:g-prob} and \eqref{est:h-prob} into \eqref{DeMorgan}, we conclude that
\begin{align}
\PP(A^c_g) & \lesssim M \left ( \exp{\left (- N \| \rho \|_{L^\infty}  \frac{r^2 \gamma^2}{2(\| g \|_{L^2}^2 + r \gamma \| g \|_{L^\infty}/3) } \right )}  +  \exp{\left (- N \| \rho \|_{L^\infty} \frac{ \gamma^2}{2( \| h \|_{L^2}^2 + \gamma \| h \|_{L^\infty}/3) } \right )}  \right ) \\
& \lesssim (r \delta)^{-d} \left ( \exp{\left (- N \| \rho \|_{L^\infty}  \frac{r^2 \gamma^2}{2(\| g \|_{L^2}^2 + r \gamma \| g \|_{L^\infty}/3 ) } \right )}  +  \exp{\left (- N \| \rho \|_{L^\infty} \frac{ \gamma^2}{2 ( \| h \|_{L^2}^2 + \gamma \| h \|_{L^\infty}/3 ) } \right )}  \right ) .
\end{align}

\end{proof}

\subsection{Law of Large Numbers for Perturbed IID Random Variables}

To conclude this section, we combine Lemma~\ref{lem:gApproxEmpirical} and Proposition~\ref{prop:LLN} to obtain the following convergence result for empirical measures $\mu_{\bf X}$ where the random variables $\{ X_i \}_{i=1}^N$ are not themselves independent, but are uniformly close to an iid $N$-tuple ${ \bf Y }$. Such is the case when we consider the particle system \eqref{eq:ODE-N} in the mean-field limit.

\begin{cor} \label{cor:LLN-Approx}
Let $\rho, {\bf Y}, r$, and the functions $(g,h)$ satisfy the assumptions of Proposition~\ref{prop:LLN}.
Assume further that $h$ has an $r$-local Lipschitz modulus $l \in L^\infty(\TT^d)$.

Let ${\bf X} : \Omega \to (\TT^d)^N$ be a random $N$-tuple of points in $\TT^d$.
For any $\delta, \gamma \in (0,1)$, let $B_g = B_g(\gamma, \delta) \in \cF$ denote the event
\be
B_g : = \Big \{ \| g \ast \mu_{{\bf X}} - g \ast \rho \|_{L^\infty }  \leq 2 \| \rho \|_{L^\infty} \left [ (\gamma + \| h  \|_{L^1} + r \| l  \|_{L^1})  |{\bf X}-{\bf Y}|_\infty +  r \left (  \delta \| h  \|_{L^1} +  \gamma \right ) \right ] \Big \} .
\ee
Then
\begin{multline}
\PP( \{ |{\bf X}-{\bf Y}|_\infty < r  \} \cap B_g^c) \lesssim  (r \delta)^{-d}  \exp{\left (- N \| \rho \|_{L^\infty} \frac{ \gamma^2}{2 ( \| r^{-1} g  \|_{L^2}^2 +  \gamma \| r^{-1 }g \|_{L^\infty}/3 ) } \right )} \\ 
+ (r \delta)^{-d} \left (  \exp{\left (- N \| \rho \|_{L^\infty}  \frac{ \gamma^2}{2( \| h  \|_{L^2}^2 + \gamma \| h \|_{L^\infty} /3 )} \right )} +   \exp{\left (- N \| \rho \|_{L^\infty} \frac{ \gamma^2}{2 ( \| r  l  \|_{L^2}^2 + \gamma \| rl \|_{L^\infty}/3 ) }   \right )} \right ) .
\end{multline} 
\end{cor}
\begin{proof}
Apply the $L^\infty$ triangle inequality to obtain
\be
\| g \ast \mu_{{\bf X}} - g \ast \rho \|_{L^\infty } \leq \| g \ast \mu_{{\bf X}} - g \ast \mu_{{\bf Y}}\|_{L^\infty } + \| g \ast \mu_{{\bf Y}} - g \ast \rho \|_{L^\infty }.
\ee
On the event $\{ |{\bf X}-{\bf Y}|_\infty < r  \}$, by Lemma~\ref{lem:gApproxEmpirical},
\be
 \| g \ast \mu_{{\bf X}} - g \ast \mu_{{\bf Y}}\|_{L^\infty } \leq \| h \ast \mu_{{\bf Y}} \|_{L^\infty} |{\bf X}-{\bf Y}|_\infty .
\ee
We approximate $h \ast \mu_{{\bf Y}}$ by $h \ast \rho$, writing
\be
 \| g \ast \mu_{{\bf X}} - g \ast \mu_{{\bf Y}}\|_{L^\infty } \leq \left ( \| h \ast \rho \|_{L^\infty} + \| h \ast \mu_{{\bf Y}} - h \ast \rho \|_{L^\infty } \right )|{\bf X}-{\bf Y}|_\infty .
\ee
Hence, applying Young's inequality for convolutions, we deduce that
\be
\| g \ast \mu_{{\bf X}} - g \ast \rho \|_{L^\infty } \leq \| h \|_{L^1} \| \rho \|_{L^\infty} |{\bf X}-{\bf Y}|_\infty + \| g \ast \mu_{{\bf Y}} - g \ast \rho \|_{L^\infty } +  \| h \ast \mu_{{\bf Y}} - h \ast \rho \|_{L^\infty } |{\bf X}-{\bf Y}|_\infty .
\ee

Now consider the event $A_g(\gamma, \delta) \cap A_h(\gamma, \delta)$, in the notation of Proposition~\ref{prop:LLN}, upon which
\begin{align}
\| g \ast \mu_{{\bf Y}} - g \ast \rho \|_{L^\infty } & \leq 2 r  \| \rho \|_{L^\infty} (\delta \| h  \|_{L^1} + \gamma) \\
\| h \ast \mu_{{\bf Y}} - h \ast \rho \|_{L^\infty } & \leq 2 r \| \rho \|_{L^\infty} (\delta \| l  \|_{L^1} + \gamma) .
\end{align}
Hence, on the event $\{ |{\bf X}-{\bf Y}|_\infty < r  \} \cap A_g(\gamma, \delta) \cap A_h(\gamma, \delta)$ we have
\begin{align}
\| g \ast \mu_{{\bf X}} - g \ast \rho \|_{L^\infty } & \leq \| h \|_{L^1} \| \rho \|_{L^\infty} |{\bf X}-{\bf Y}|_\infty + 2 r  \| \rho \|_{L^\infty} (\delta \| h  \|_{L^1} + \gamma) + 2 r \| \rho \|_{L^\infty} (\delta \| l  \|_{L^1} + \gamma)  |{\bf X}-{\bf Y}|_\infty \\
& \leq 2 \| \rho \|_{L^\infty} \left [ ( \| h \|_{L^1} + r \| l  \|_{L^1} + \gamma )  |{\bf X}-{\bf Y}|_\infty +  r \left (  \delta \| h \|_{L^1} +  \gamma \right ) \right ] .
\end{align}
Thus $\{ |{\bf X}-{\bf Y}|_\infty < r  \} \cap A_g(\gamma, \delta) \cap A_h(\gamma, \delta) \subset B_g$.

Therefore, $\{ |{\bf X}-{\bf Y}|_\infty < r  \} \cap B_g^c \subset A_g(\gamma, \delta)^c \cap A_h(\gamma, \delta)^c$, and
we may estimate
\be
\PP(\{ |{\bf X}-{\bf Y}|_\infty < r  \} \cap B^c_g) \leq \PP(A_g^c) + \PP(A_h^c) .
\ee
The conclusion then follows from Proposition~\ref{prop:LLN}.

\end{proof}

\section{Mean-Field Limit for Ionic Vlasov-Poisson}
\label{sec:MFL}

\subsection{Nonlinear Kinetic `Distance'} \label{sec:KineticDistance}

In this section, we present the proof of Theorem~\ref{thm:main}. The goal is to compare the solutions $({\bf X}^N, {\bf V}^N)$ and $({\bf Y}^N, {\bf W}^N)$ of the ODEs
\be \label{eq:CompareODE}
\begin{cases}
\dot X_i^N = V_i^N \\
\dot V_i^N = - 
\nabla_x \Phi_r [\mu_{{\bf X}}] (X_i^N) ,
\end{cases}
\qquad
\begin{cases}
\dot Y_i^N = W_i^N \\
\dot W_i^N = - 
\nabla_x \Phi_r [ \rho_{f_r}] (Y_i^N ) ,
\end{cases}
\quad i \in \{ 1 , \ldots, N \} .
\ee
A technical difficulty in obtaining these estimates is that the (uniform in $r$) regularity estimates on the electric field are of \emph{log-Lipschitz} rather than Lipschitz type (recall Proposition~\ref{prop:potential}). 
Previous works \cite{LazaroviciPickl, Lazarovici} on the electron model used the regularisation (or truncation) of the kernel $K_r$ to obtain that the regularised electric field is Lipschitz with Lipschitz constant diverging at the rate $|\log r|$. By combining this with an anisotropically weighted distance, e.g. of the form
\be
|\log r |^{1/2} | {\bf X}^N - {\bf Y}^N |_\infty + |{\bf V}^N - {\bf W}^N |_\infty ,
\ee
it was possible to obtain a Gr\"{o}nwall-type estimate with constant diverging at a rate slower than any algebraic rate in $r$.
A similar approach was used for the ion model in \cite{GPIMFQN}, by employing a doubled regularisation method so that the field was of the form $K_r \ast (\rho_{f_r} - e^{\Phi_r[\rho_{f_r}]})$.

In the particle system \eqref{eq:CompareODE}, the regularised interaction potentials $\Phi_r$ solve equations of type
\be
- \Delta \Phi_r[\mu] = \chi_r \ast \mu - e^{\Phi_r[\mu]}
\ee
and are therefore not of the form of a convolution with $K_r$.
Thus, to use a similar approach we would need more involved estimates on $e^{\Phi_r[\mu]}$ in order to obtain Lipschitz bounds on $\nabla \Phi_r[\mu]$ diverging at the correct rate.
In the present work we avoid this by using the log-Lipschitz estimates directly. 
The doubled regularisation is then no longer required, although our results would also hold in this case.

We make use of a technique developed by Iacobelli \cite{Iacobelli} for the purpose of obtaining sharper quantitative Wasserstein stability estimates for kinetic transport equations with non-Lipschitz force fields.
The idea is to control the distance between the particle systems using an anisotropically weighted functional where the weight \emph{depends nonlinearly upon the quantity itself}.

\begin{definition} \label{def:J}
Let $0 < r, \alpha_0, \beta_0 < \frac{1}{3e}$ be given. $J : [0, + \infty) \to [0, + \infty)$ is a function satisfying
\be \label{JProperty}
J(t) : = (r + \alpha_0 + \beta_0 ) \wedge \left (  \sqrt{|\log J(t) |} \, \left (  \|{\bf X}^N - {\bf Y}^N \|_{L^\infty(0,t)} + \alpha_0 \right )+   \|{\bf V}^N - {\bf W}^N \|_{L^\infty(0,t)}  + \beta_0  \right ) .
\ee 

\end{definition}

Here $r$ denotes the regularisation parameter as before, while $\alpha_0$ and $\beta_0$ will be chosen later.
We will defer the proof that $J$ exists until later in the subsection. We first note the following properties of $J$, which show that $J$ indeed provides upper bounds for $\|{\bf X}^N - {\bf Y}^N \|_{L^\infty(0,t)} $ and $\|{\bf V}^N - {\bf W}^N \|_{L^\infty(0,t)} $. Moreover, as in \cite{BoersPickl, LazaroviciPickl}, by truncating $J$ at $r + \alpha_0 + \beta_0$ we ensure that $J$ is non-trivial only if $\|{\bf X}^N - {\bf Y}^N \|_{L^\infty(0,t)} $ is below the threshold $r$; when this holds, the estimates between empirical measures from Section~\ref{sec:LLN} using the $r$-local Lipschitz bounds can be used.

\begin{lemma} \label{lem:Jbasics}
If $J$ satisfies \eqref{JProperty}, then:
\begin{enumerate}[(i)]
\item $J \leq r + \alpha_0 + \beta_0 < \frac{1}{e}$ \label{item:J-UB}
\item $|\log J| \geq 1$ \label{item:Jlog}
\item If $J < r + \alpha_0 + \beta_0$, then the following inequalities hold: \label{item:Jineqs}
\begin{align}
&\|{\bf X}^N - {\bf Y}^N \|_{L^\infty(0,t)} +  \|{\bf V}^N - {\bf W}^N \|_{L^\infty(0,t)} < r \, ; \\
&\|{\bf X}^N - {\bf Y}^N \|_{L^\infty(0,t)}  + \alpha_0  \leq \frac{J}{|\log J|^{1/2}} \, ; \\
&\|{\bf V}^N - {\bf W}^N \|_{L^\infty(0,t)}  + \beta_0 \leq J .
\end{align}

\end{enumerate}
\end{lemma}
\begin{proof}
\eqref{item:J-UB} follows immediately from the definition of minimum and the assumption that $r, \alpha_0, \beta_0 < \frac{1}{3e}$.
Then, since $J < \frac{1}{e} < 1$, 
\be
|\log J| = - \log J > - \log \frac{1}{e} = 1 ,
\ee
and \eqref{item:Jlog} follows.

For \eqref{item:Jineqs}, note that when $J < r + \alpha_0 + \beta_0$,
\be
J = \sqrt{|\log J |} \, \left ( \|{\bf X}^N - {\bf Y}^N \|_{L^\infty(0,t)} + \alpha_0 \right )+  \|{\bf V}^N - {\bf W}^N \|_{L^\infty(0,t)} + \beta_0 .
\ee
Hence we obtain
\be
\|{\bf V}^N - {\bf W}^N \|_{L^\infty(0,t)}  + \beta_0 \leq J, \quad  \|{\bf X}^N - {\bf Y}^N \|_{L^\infty(0,t)} + \alpha_0  \leq J |\log J|^{-1/2}
\ee
as required.

Finally, \eqref{item:Jlog} implies that
\be
\|{\bf X}^N - {\bf Y}^N \|_{L^\infty(0,t)}  +\|{\bf V}^N - {\bf W}^N \|_{L^\infty(0,t)}  + \alpha_0 + \beta_0 \leq J < r + \alpha_0 + \beta_0 .
\ee
Hence
\be
\|{\bf X}^N - {\bf Y}^N \|_{L^\infty(0,t)} + \|{\bf V}^N - {\bf W}^N \|_{L^\infty(0,t)}  < r,
\ee
and the conclusion follows.

\end{proof}

In the remainder of this subsection, we establish the existence of $J$, and show that it is Lipschitz continuous as a function of $t$.
This can be done using the methods of \cite[Lemma 3.7]{Iacobelli}. First, we note the following lemma, which is a consequence of the implicit function theorem.

\begin{lemma} \label{lem:IFT}
Let $F: (0, 1) \times (0,1) \times (0,1) \to \R$ be the function
\be
F(a,b ,j) = j - a |\log j|^{1/2} - b .
\ee
For all $a \in (0,1)$ and $b \in (0,\frac{1}{e})$, there exists a unique $j_\ast \in (0, 1)$ such that $F(a,b, j_\ast) = 0$. Moreover, there exists a $C^1$ function $I : (0,1) \times (0,1) \to (0,1)$ such that $j_\ast = I(a,b)$.
\end{lemma}

Using Lemma~\ref{lem:IFT} we may deduce that the quantity $J$ exists and is differentiable almost everywhere with respect to $t$.

\begin{lemma} \label{lem:J_Lip}
There exists a Lipschitz function $J: [0, +\infty) \to (0,r + \alpha_0 + \beta_0)$ satisfying Definition~\ref{def:J}.

\end{lemma}

\begin{proof}
Let $I$ be the implicit function defined in Lemma~\ref{lem:IFT}. Let
\be 
J : = \begin{cases}
(r + \alpha_0 + \beta_0) \wedge I \left (\|{\bf X}^N- {\bf Y}^N\|_{L^\infty(0,t)}  + \alpha_0, \|{\bf V}^N- {\bf W}^N\|_{L^\infty(0,t)} + \beta_0 \right ) & \text{if }  \|{\bf X}^N- {\bf Y}^N\|_{L^\infty(0,t)} , \|{\bf V}^N- {\bf W}^N\|_{L^\infty(0,t)}  \leq r \\
r + \alpha_0 + \beta_0 & \text{otherwise.}
\end{cases}
\ee
By Lemma~\ref{lem:IFT}, this is well defined, since $\|{\bf X}^N- {\bf Y}^N\|_{L^\infty(0,t)}  + \alpha_0 < 1$ and $\|{\bf V}^N- {\bf W}^N\|_{L^\infty(0,t)}  + \beta_0 < \frac{1}{e}$ when ${\|{\bf X}^N- {\bf Y}^N\|_{L^\infty(0,t)} }$, ${\|{\bf V}^N- {\bf W}^N\|_{L^\infty(0,t)}} \leq r < \frac{1}{e}$.
$J$ clearly satisfies Definition~\ref{def:J} when $\|{\bf X}^N- {\bf Y}^N\|_{L^\infty(0,t)} , \|{\bf V}^N- {\bf W}^N\|_{L^\infty(0,t)}  \leq r$, by the definition of $I$. However, by Lemma~\ref{lem:Jbasics}, if $\|{\bf X}^N- {\bf Y}^N\|_{L^\infty(0,t)}  > r$ or $\|{\bf V}^N- {\bf W}^N\|_{L^\infty(0,t)}  > r$ then property \eqref{JProperty} is merely the statement that $J = r + \alpha_0 + \beta_0$. Hence $J$ satisfies Definition~\ref{def:J}.

Next, we show that $J$ is Lipschitz as a function of time. Both of the functions $t \mapsto |{\bf X}^N(t) - {\bf Y}^N(t) |_{\infty}$ and $t \mapsto |{\bf V}^N(t) - {\bf W}(t) |_{\infty} $ are Lipschitz, each being the maximum of a finite number of Lipschitz functions.
Then the supremum functions $t \mapsto \|{\bf X}^N- {\bf Y}^N\|_{L^\infty(0,t)}$ and $t \mapsto \|{\bf V}^N- {\bf W}^N\|_{L^\infty(0,t)} $ are also Lipschitz.

Let 
\be
t ' : = \sup \{ t \geq 0 : \|{\bf X}^N- {\bf Y}^N\|_{L^\infty(0,t)} + \|{\bf V}^N- {\bf W}^N\|_{L^\infty(0,t)}  \leq r \} .
\ee
Then $I \left ( \|{\bf X}^N- {\bf Y}^N\|_{L^\infty(0,t)} + \alpha_0, \|{\bf V}^N- {\bf W}^N\|_{L^\infty(0,t)}  + \beta_0 \right )$ is Lipschitz for $t \leq t'$, as the composition of a $C^1$ function with a Lipschitz function.

It follows that $I \left ( \|{\bf X}^N- {\bf Y}^N\|_{L^\infty(0,t)} + \alpha_0, \|{\bf V}^N- {\bf W}^N\|_{L^\infty(0,t)}  + \beta_0 \right )$  is Lipschitz for $t$ such that $\|{\bf X}^N- {\bf Y}^N\|_{L^\infty(0,t)} + \|{\bf V}^N- {\bf W}^N\|_{L^\infty(0,t)}  < r$, since it is the composition of a $C^1$ function with a Lipschitz function. The truncation occurs at the time $t_\ast$, defined by
\be
t_\ast : = \inf \{ t \geq 0 :  I \left (\|{\bf X}^N- {\bf Y}^N\|_{L^\infty(0,t)}  + \alpha_0, \|{\bf V}^N- {\bf W}^N\|_{L^\infty(0,t)} + \beta_0 \right )  \geq r + \alpha_0 + \beta_0 \} ,
\ee
at which point
\be
J(t_\ast) = r + \alpha_0 + \beta_0 = \left (  \sqrt{|\log J(t_\ast) |} \, \left (\alpha_0 +  \|{\bf X}^N- {\bf Y}^N\|_{L^\infty(0,t_\ast)}  \right )+   \|{\bf V}^N- {\bf W}^N\|_{L^\infty(0,t_\ast)}  + \beta_0  \right ).
\ee
Since $|\log J(t_\ast) | > 1$, we have
\be
  \|{\bf X}^N- {\bf Y}^N\|_{L^\infty(0,t_\ast)}  +   \|{\bf V}^N- {\bf W}^N\|_{L^\infty(0,t_\ast)} < r,
\ee
and hence $t_\ast < t'$. Thus, for $t \in [0, t']$, $J(t)$ is a maximum of Lipschitz functions and therefore Lipschitz. Moreover $J(t) = r + \alpha_0 + \beta_0$ for all $t \in [t_\ast, t']$, so that the extension $J(t) = r + \alpha_0 + \beta_0$ for $t > t '$ defines a Lipschitz function for all $t$.

\end{proof}

\subsection{Dynamical Estimate}

Having constructed the quantity $J$, we now turn to estimating it via a Gr{\"{o}}nwall-type argument. 

\begin{paragraph}{Note on Constants.}
In this subsection, many generic constants $C>0$ depend on the dimension $d$ and choice of mollifier $\chi$. Since $d$ and $\chi$ are fixed throughout the argument, in order to lighten the notation we will not indicate this dependence explicitly with the notations $C_{d, \chi}$, $\lesssim_{d, \chi}$ etc., instead abbreviating these to $C$, $\lesssim$.

\end{paragraph}

\subsubsection{A First Integral Inequality for $J$}

\begin{lemma} \label{lem:J_integral}
Let $J$ be as in Definition~\ref{def:J}. Then
\be
- |\log J(t)|^{1/2} \leq - | \log J (0) |^{1/2} + \frac{1}{2}  \int_0^t \left ( 1 + \frac{ \max_i | E^{{\bf X}}(s, X_i^N(s)) - E^f(s, Y_i^N(s))|}{J(s) | \log J (s) |^{1/2} } \right )  \one_{\{ J(s) < r + \alpha_0 + \beta_0 \}}  \dd s .
\ee
\end{lemma}

\begin{proof}
By Lemma~\ref{lem:J_Lip}, $J$ is Lipschitz as a function of $t$. Since $J$ takes values in the interval $( \alpha_0 + \beta_0 , \frac{1}{e}) $, then $- | \log J |^{1/2}$ is a $C^1$ function of a Lipschitz function and therefore Lipschitz. Hence $- | \log J |^{1/2}$ is differentiable for almost all $t$, and we may represent it in the form
\be
 - | \log J (t) |^{1/2}  =  - | \log J (0) |^{1/2} - \int_0^t \frac{\dd }{\dd s} | \log J (s) |^{1/2}  \dd s .
\ee
For all $s \in [0,t]$ such that $J$ is differentiable at $s$, we may apply the chain rule to obtain $- \frac{\dd }{\dd s} | \log J (s) |^{1/2} = \frac{\dot J (s) }{2 J(s) | \log J (s) |^{1/2} }$. Since this holds for almost all $s \in [0,t]$, we have
\be
 - | \log J (t) |^{1/2} = - | \log J (0) |^{1/2} + \int_0^t \frac{\dot J (s) }{2 J(s) | \log J (s) |^{1/2} } \dd s .
\ee
As we are seeking an upper bound on $- |\log J|^{1/2}$, it will suffice to consider only $s$ for which $\dot J \geq 0$.
We denote the positive part of $\dot J$ by $\left [\dot J \right ]_+$, and deduce that
\be \label{est:loghalf_1}
 - | \log J (t) |^{1/2} \leq - | \log J (0) |^{1/2} + \int_0^t \frac{[\dot J (s) ]_+}{2 J(s) | \log J (s) |^{1/2} } \one_{\{ J(s) < r + \alpha_0 + \beta_0 \}}  \dd s .
\ee
For $s$ such that $J(s) < r + \alpha_0 + \beta_0$, we have the relation
\be \label{eq:J_chain_relation}
J(s) : = |\log J(s) |^{1/2} \, \left (  \|{\bf X}^N- {\bf Y}^N\|_{L^\infty(0,s)} + \alpha_0 \right )+   \|{\bf V}^N- {\bf W}^N\|_{L^\infty(0,s)}  + \beta_0  .
\ee 
For almost all such $s$, $J$ is differentiable at $s$, and we may apply the chain rule to \eqref{eq:J_chain_relation} to obtain
\be
\dot J(s) = - (\dot J(s)) \frac{ \|{\bf X}^N- {\bf Y}^N\|_{L^\infty(0,s)} + \alpha_0 }{2 J(s) |\log J(s)|^{1/2}} + |\log J(s) |^{1/2} \frac{\dd }{\dd s} \|{\bf X}^N- {\bf Y}^N\|_{L^\infty(0,s)}  + \frac{\dd }{\dd s} \|{\bf V}^N- {\bf W}^N\|_{L^\infty(0,s)} .
\ee
When $\dot J \geq 0$, the first term is non-positive. Hence, for almost all $s \in [0,t]$,
\begin{align}
\left [ \dot J(s) \right ]_+ & \leq \left (  |\log J(s) |^{1/2} \frac{\dd }{\dd s} \|{\bf X}^N- {\bf Y}^N\|_{L^\infty(0,s)} + \frac{\dd }{\dd s} \|{\bf V}^N- {\bf W}^N\|_{L^\infty(0,s)} \right ) \one_{\{ \dot J(s) \geq 0 \}} \\
& \leq  |\log J(s) |^{1/2} \left [ \frac{\dd }{\dd s} |{\bf X}^N(s) - {\bf Y}^N(s) |_{\infty} \right ]_+ + \left [  \frac{\dd }{\dd s} |{\bf V}^N(s) - {\bf W}^N(s) |_{\infty} \right ]_+ .
\end{align}

Next, we estimate the derivatives $\left [ \frac{\dd }{\dd s} |{\bf X}^N(s) - {\bf Y}^N(s) |_{\infty} \right ]_+$ and $\left [  \frac{\dd }{\dd s} |{\bf V}^N(s) - {\bf W}^N(s) |_{\infty} \right ]_+$.
We take the difference of the ODE systems \eqref{eq:CompareODE} to obtain
\be \label{eq:ODEDiff}
\begin{cases}
\frac{\dd }{\dd t} (X_i^N - Y_i^N) = V_i^N-  W_i^N \\
\frac{\dd }{\dd t} (V_i^N-  W_i^N) = - \nabla_x \Phi_r [\mu_{{\bf X}^N}] (X_i^N) + \nabla_x \Phi_r [ \rho_{f_r}] (Y_i^N ) .
\end{cases}
\ee
We introduce the shorthand notations
\be
E^{{\bf X}} : = - \nabla_x \Phi_r [ \mu_{{\bf X}^N}] 
\qquad E^f : = - \nabla_x \Phi_r [ \rho_{f_r}] .
\ee
It then follows from \eqref{eq:ODEDiff} that, for almost all $s\in [0,t]$,
\be
\frac{\dd }{\dd s} |{\bf X}^N(s) - {\bf Y}^N(s) |_{\infty} \leq |{\bf V}^N(s) -{\bf W}^N(s)|_\infty \quad  \frac{\dd }{\dd s} |{\bf V}^N(s) - {\bf W}^N(s) |_{\infty}  \leq \max_i | E^{ {\bf X}}(s,X_i^N(s)) - E^f(s, Y_i^N(s))| .
\ee
Hence
\be
\left [ \dot J(s) \right ]_+ \leq |\log J(s) |^{1/2} \|{\bf V}^N- {\bf W}^N\|_{L^\infty(0,s)}  + \max_i | E^{{\bf X}}(s, X_i^N(s)) - E^f(s, Y_i^N(s))| .
\ee
By Lemma~\ref{lem:Jbasics}\eqref{item:Jineqs}, if $\dot J(s) > 0$ then $J(s) < r + \alpha_0 + \beta_0$, and hence $\|{\bf V}^N- {\bf W}^N\|_{L^\infty(0,s)}  \leq J(s)$. Thus
\be \label{est:JdotPlus}
\left [ \dot J \right ]_+ \leq |\log J |^{1/2} J + \max_i | E^{{\bf X}}(X_i^N) - E^f(Y_i^N)| \qquad a.e.
\ee
We conclude that, by \eqref{est:loghalf_1},
\begin{align} 
 - | \log J (t) |^{1/2} & \leq - | \log J (0) |^{1/2} + \frac{1}{2}  \int_0^t \left ( 1 + \frac{ \max_i | E^{{\bf X}}(s, X_i^N(s)) - E^f(s, Y_i^N(s))|}{J(s) | \log J (s) |^{1/2} } \right )  \one_{\{ J(s) < r + \alpha_0 + \beta_0 \}}  \dd s \\ \label{est:loghalf_2}
& \leq - | \log J (0) |^{1/2} + \frac{1}{2}  \int_0^t \left ( 1 + \frac{ \max_i | E^{{\bf X}}(s, X_i^N(s)) - E^f(s, Y_i^N(s))|}{J(s) | \log J (s) |^{1/2} } \right )   \one_{\{  \|{\bf X}^N- {\bf Y}^N\|_{L^\infty(0,s)} < r \}}  \dd s .
\end{align}

\end{proof}

\noindent The next step is to obtain an estimate for the term $\max_i | E^{{\bf X}}(X_i^N) - E^f(Y_i^N)|$, on the event ${\{ \|{\bf X}^N- {\bf Y}^N\|_{L^\infty(0,s)}< r \}}$.

\begin{prop} \label{prop:Field}
Let $\gamma, \delta \in (0,1)$. Set $\alpha_0 = r \delta$, $\beta_0 = r \gamma$, and let J be defined as in Definition~\ref{def:J}.
There exists a non-decreasing function $H: [0, + \infty) \to [0, + \infty)$ such that the following holds.

Suppose that Assumption~\ref{hyp:f} holds: there exists a function $D: [0, T) \to [0, + \infty)$ such that, for all $t \in [0,T)$,
\be
\sup_{r \leq \frac{1}{2}, s \in [0,t]} \| \rho_{f_r}(s, \cdot) \|_{L^\infty(\TT^d)} \leq D(t).
\ee 
For all $t \in [0,T)$, let $\Gamma = \Gamma(t, r, \delta, \gamma)$ denote the event
\be
\Gamma : = \left \{ \max_i | E^{{\bf X}}(X_i^N(t)) - E^f(Y_i^N(t))| \leq H(D(t)) |\log J(t)|^{1/2} J(t) \right \} .
\ee
Then $\Gamma$ satisfies
\be
\PP(\Gamma^c \cap\{ \|{\bf X}^N - {\bf Y}^N\|_{L^\infty(0,t)} < r \}) \lesssim (r \delta)^{-d} \exp \left (- C N r^d \frac{\gamma^2}{1 + \gamma} \right ) .
\ee
Moreover, on the event $\Gamma^c$,
\be \label{est:NotGamma}
\max_i | E^{{\bf X}}(X_i^N(t)) - E^f(Y_i^N(t))| \lesssim r^{-d}.
\ee

\end{prop}

\begin{proof}
Throughout this proof, $t$ is fixed: the $t$-dependence of quantities is therefore suppressed in the notation where possible. 

We begin with the final statement \eqref{est:NotGamma}: we will prove  an $L^\infty$ estimate for $E^{{\bf X}}$ and $E^f$ that is always true, and hence in particular holds on $\Gamma^c$.
For any probability measure $\mu \in \mc{P}(\TT^d)$, the solution $\Phi_r[\mu]$ of the equation
\be
- \Delta \Phi_r[\mu] = \chi_r \ast \mu - e^{\Phi_r[\mu]} .
\ee
satisfies
\be
\nabla \Phi_r[\mu] = K \ast (\chi_r \ast \mu - e^{\Phi_r[ \mu ]}) .
\ee
By Proposition~\ref{prop:potential},
\be
\| \nabla \Phi_r[\mu] \|_{L^\infty} \leq \| K \|_{L^1} \| \chi_r \ast \mu - e^{\Phi_r[ \mu ]} \|_{L^\infty} \leq \| K \|_{L^1} \| \chi_r \ast \mu \|_{L^\infty} .
\ee
Then, by Lemma~\ref{lem:Fn_bounds}, $\| \chi_r \ast \mu \|_{L^\infty} \lesssim r^{-d}$. Hence
\be
\| \nabla \Phi_r[\mu] \|_{L^\infty} \lesssim r^{-d} .
\ee
It follows that, almost surely
\be
\max_i | E^{{\bf X}}(X_i^N) - E^f(Y_i^N)| \lesssim r^{-d} .
\ee

Now we turn to the analysis of the event $\Gamma$. The strategy is as follows: we will identify an event $\widetilde \Gamma$ satisfying 
\be \label{est:GammaTildeProb}
\PP \left ( \widetilde \Gamma^c \cap \{ \|{\bf X}^N - {\bf Y}^N\|_{L^\infty(0,t)} < r \} \right ) \lesssim  (r \delta)^{-d} \exp \left (- C N r^d \frac{\gamma^2}{1 + \gamma} \right )
\ee 
such that, if we make a suitable choice of the function $H$, we have the inclusion $\widetilde \Gamma \subset \Gamma$. If this is the case, we will then be able to deduce that
\be
\PP \left ( \Gamma^c \cap \{ \|{\bf X}^N - {\bf Y}^N\|_{L^\infty(0,t)} < r \} \right ) \leq \PP (\widetilde \Gamma^c \cap \{ \|{\bf X}^N - {\bf Y}^N\|_{L^\infty(0,t)} < r \}) \lesssim (r \delta)^{-d} \exp \left ( - C N r^d \frac{\gamma^2}{1 + \gamma} \right )
\ee
as desired. To find $\widetilde \Gamma$, we use the law of large numbers estimates from Section~\ref{sec:LLN} and look for a more refined bound on $\max_i | E^{{\bf X}}(X_i^N) - E^f(Y_i^N)|$.

Using the triangle inequality, for each $i=1, \ldots, N$ we write
\be \label{est:triangle}
| E^{{\bf X}}(X_i^N) - E^f(Y_i^N)| \leq | E^{{\bf X}}(X_i^N) - E^f(X_i^N)|
 + | E^f(X_i^N) -E^f(Y_i^N)| .
\ee
We estimate each term in turn.

For the term $ | E^f(X_i^N) -E^f(Y_i^N)|$, by the log-Lipschitz estimate (Proposition~\ref{prop:potential}) and the uniform $L^\infty$ estimate \eqref{est:rho_unif} on $\rho_{f_r}$, we have
\be \label{est:EReg}
 | E^f(X_i^N) -E^f(Y_i^N )|  \lesssim D \, |X_i^N - Y_i^N| \log | X_i^N - Y_i^N |^{-1} \lesssim D \, \|{\bf X}^N-{\bf Y}^N\|_{L^\infty(0,t)} \log \|{\bf X}^N-{\bf Y}^N\|_{L^\infty(0,t)}^{-1}
 \ee
 where the second inequality uses the fact that the function $x \mapsto x|\log x|$ is increasing for $x \in (0, \frac{1}{e})$.

For the term $ | E^{{\bf X}}(X_i^N) - E^f(X_i^N)|$ of \eqref{est:triangle}, we apply Proposition~\ref{prop:PotentialStability} in order to estimate $\left \| E^{{\bf X}} - E^f \right \|_{L^\infty}$ as follows, recalling that $K_r = K \ast \chi_r$: 
\be
\left \| E^{{\bf X}} - E^f \right \|_{L^\infty}  
\leq  \| K_r \ast \mu_{{\bf X}^N} - K_r \ast \rho_{f_r} \|_{L^\infty} \, \exp \left ( C \max \{ \| \chi_r \ast \mu_{{\bf X}^N} \|_{L^\infty}, \| \chi_r \ast \rho_{f_r} \|_{L^\infty} \} \right ) . \label{est:EDiffLinfty}
\ee
We note that, by Assumption~\ref{hyp:f},
\be \label{est:rho_unif}
\| \chi_r \ast \rho_{f_r} \|_{L^\infty} \leq D \quad \text{for all } r \in (0, \frac{1}{4} ) .
\ee 
It therefore suffices to obtain suitable estimates on $\| \chi_r \ast \mu_{{\bf X}^N} - \chi_r \ast  \rho_{f_r} \|_{L^\infty}$ and $\| K_r \ast \mu_{{\bf X}^N} - K_r \ast \rho_{f_r} \|_{L^\infty}$.

With this in mind, we apply Corollary~\ref{cor:LLN-Approx} to the $N$-tuples ${\bf X} = {\bf X}^N(t)$ and ${\bf Y}= {\bf Y}^N(t)$, for the cases $g = K_r^{(j)}$ for each coordinate $j=1, \ldots, d$ and $g = r \chi_r$. We obtain the events $B_{K_r^{(j)}}(\gamma, \delta)$ and $B_{r \chi_r}(\gamma, \delta)$. We then define the event
\be
\widetilde \Gamma : = \{ \|{\bf X}^N - {\bf Y}^N\|_{L^\infty(0,t)} < r \} \cap \bigcap_{j=1}^d B_{K_r^{(j)}}(\gamma, \delta) \cap B_{r \chi_r}(\gamma, \delta).
\ee 
Then $\widetilde \Gamma$ satisfies \eqref{est:GammaTildeProb}: by Definition~\ref{def:moduli}, Lemma~\ref{lem:Fn_bounds} and Corollary~\ref{cor:LLN-Approx}, we have for each $j$
\begin{align}
\PP (\{ |{\bf X}^N-{\bf Y}^N|_\infty < r  \} \cap B_{K_r^{(j)}}(\gamma, \delta)^c) & \lesssim  (r \delta)^{-d}  \exp{\left (- C N r^d \| \rho_{f_r(s)} \|_{L^\infty} \frac{ \gamma^2}{ 1 +  \gamma  } \right )}  \\
& \lesssim (r \delta)^{-d} \exp \left (- C N r^d \frac{\gamma^2}{1 + \gamma} \right ),
\end{align}
since $1 = \| \rho_{f_r} \|_{L^1(\TT^d)} \lesssim \| \rho_{f_r} \|_{L^\infty(\TT^d)} $ as the torus has finite volume. Similarly
\be
\PP (\{ |{\bf X}^N-{\bf Y}^N|_\infty < r  \} \cap B_{r \chi_r}(\gamma, \delta)^c) \lesssim (r \delta)^{-d} \exp \left (- C N r^d \frac{\gamma^2}{1 + \gamma} \right ),
\ee
and \eqref{est:GammaTildeProb} follows by summation, since \be
\widetilde \Gamma^c \cap \{ \|{\bf X}^N - {\bf Y}^N\|_{L^\infty(0,t)} < r \} \subset \{ |{\bf X}^N-{\bf Y}^N|_\infty < r  \} \cap \left( \bigcup_{j=1}^d B_{K_r^{(j)}}(\gamma, \delta)^c \cup B_{r \chi_r}(\gamma, \delta)^c \right ).
\ee

We continue our estimation of $\left \| E^{{\bf X}} - E^f \right \|_{L^\infty}$ on the event $\widetilde \Gamma$.
Using the definitions of $B_{K_r^{(j)}}(\gamma, \delta)$ and $B_{r \chi_r}(\gamma, \delta)$ in Corollary~\ref{cor:LLN-Approx}, on $\widetilde \Gamma$ we have
\begin{align}
& \| K_r \ast \mu_{{\bf X}^N} - K_r \ast \rho_{f_r} \|_{L^\infty} \lesssim  D \left [ (\gamma + \| L_r  \|_{L^1} + r \| Q_r \|_{L^1})  |{\bf X}^N-{\bf Y}^N|_\infty +  r \left (  \delta \| L_r  \|_{L^1} +  \gamma \right ) \right ] \\
& r \| \chi_r \ast \mu_{{\bf X}^N} - \chi_r \ast \rho_{f_r} \|_{L^\infty} \lesssim D \left [ (\gamma + r \| \psi_r  \|_{L^1} + r^2 \| \eta_r  \|_{L^1})  |{\bf X}^N-{\bf Y}^N|_\infty +  r \left (  \delta \| r \psi_r  \|_{L^1} +  \gamma \right ) \right ] .
\end{align}
Then, by Lemma~\ref{lem:Fn_bounds},
\begin{align}
& \| K_r \ast \mu_{{\bf X}^N} - K_r \ast \rho_{f_r} \|_{L^\infty} \lesssim D  \left [ |\log r|   (1 + \gamma ) \left (  |{\bf X}^N-{\bf Y}^N|_\infty + r \delta \right )  + r  \gamma \right ]  \\
& \| \chi_r \ast \mu_{{\bf X}^N} - \chi_r \ast \rho_{f_r} \|_{L^\infty} \lesssim D \left ( (1 + \gamma ) r^{-1} |{\bf X}^N-{\bf Y}^N|_\infty +  \delta +  \gamma \right ) .
\end{align}
Since $|{\bf X}^N-{\bf Y}^N|_\infty \leq \|{\bf X}^N-{\bf Y}^N\|_{L^\infty(0,t)} < r$,  $\delta < 1$, and $\gamma < 1$, we deduce that
\begin{align}
 \| K_r \ast \mu_{{\bf X}^N} - K_r \ast \rho_{f_r} \|_{L^\infty} & \lesssim D  \left [ |\log r|  \left ( \|{\bf X}^N-{\bf Y}^N\|_{L^\infty(0,t)} + r \delta \right )  + r  \gamma \right ]  \\
 \| \chi_r \ast \mu_{{\bf X}^N} \| _{L^\infty} & \lesssim D .
\end{align}
Hence
from \eqref{est:EDiffLinfty} we obtain
\be
\left \| E^{{\bf X}} - E^f \right \|_{L^\infty} \leq e^{C D} \left (  |\log r| \left ( \|{\bf X}^N-{\bf Y}^N\|_{L^\infty(0,t)} + r \delta \right ) + r \gamma \right ) . \label{est:EStab}
\ee

 By substituting estimates \eqref{est:EReg} and \eqref{est:EStab} into the inequality \eqref{est:triangle} and taking the maximum over $i$, we obtain that on $\widetilde \Gamma$ the following bound holds:
 \be \label{est:EDiff_XY}
\max_i  | E^{{\bf X}}(X_i^N) - E^f(Y_i^N)| \lesssim_D  |\log r| \left ( \|{\bf X}^N-{\bf Y}^N\|_{L^\infty(0,t)} + r \delta \right ) +  \|{\bf X}^N-{\bf Y}^N\|_{L^\infty(0,t)} \log \|{\bf X}^N-{\bf Y}^N\|_{L^\infty(0,t)}^{-1} + r \gamma ,
 \ee
 where the constant is of the form $e^{CD}$ for some $C>0$ independent of $D$.
 
  To complete the proof, we must show that by choosing $H$ appropriately we can ensure that $\widetilde \Gamma \subset \Gamma$. 
Our goal is therefore to bound the right hand side of \eqref{est:EDiff_XY} in terms of $J$. 
 From Lemma~\ref{lem:Jbasics} we obtain the inequalities
 \begin{align}
\|{\bf X}^N-{\bf Y}^N\|_{L^\infty(0,t)} + r \delta & < 2r \label{est:particles-J_1} \\
\|{\bf X}^N-{\bf Y}^N\|_{L^\infty(0,t)} + r \delta & \leq J |\log J|^{-1/2} \label{est:particles-J_2} \\
 r \gamma  & \leq J \label{est:particles-J_3} . 
 \end{align}
Since the function $x \mapsto |\log x|$ is decreasing for $x \in (0, \frac{1}{e})$, by inequality \eqref{est:particles-J_1} we have
\be
|\log 2r | \leq \left | \log \left ( \|{\bf X}^N-{\bf Y}^N\|_{L^\infty(0,t)} + \delta r  \right ) \right | .
\ee
However, by the same token $ |\log r| \geq |\log 2r|$. Nevertheless, since $|\log r| = |\log 2r| + \log 2$ and $|\log 2r| \geq 1$, we have $|\log r| \leq (1 + \log 2) |\log 2r|$ and hence
\be \label{est:logr-J}
|\log r |  \lesssim  \left | \log \left (  \|{\bf X}^N-{\bf Y}^N\|_{L^\infty(0,t)} + \delta r  \right ) \right | .
\ee
Since the function $x \mapsto x |\log x|$ is increasing for $x \in (0, \frac{1}{e})$, by substituting inequality \eqref{est:logr-J} into \eqref{est:EDiff_XY} and then applying inequalities \eqref{est:particles-J_2}-\eqref{est:particles-J_3} we obtain

 \be
\max_i  | E^{{\bf X}}(X_i^N) - E^f(Y_i^N)| \leq e^{C(1 +D)}  J |\log J|^{-1/2}  | \log (J |\log J|^{-1/2} ) | + J .
 \ee
 
We conclude by observing that
 \be
  | \log (J |\log J|^{-1/2} ) | \lesssim |\log J| .
 \ee
Indeed, since $J \leq \frac{1}{e}$ and thus $J |\log J|^{-1/2} \leq J \leq 1$, from the inequality $\log x \leq x - 1$ we obtain
\begin{align}
  | \log (J |\log J|^{-1/2} ) | & = - \log J + \frac{1}{2} \log (- \log J ) \leq \frac{3}{2} |\log J| .
\end{align}
 We have thus shown that, on the event $\widetilde \Gamma$,
\be
\max_i  | E^f(X_i^N) -E^f(Y_i^N )| \leq e^{C(1 + D)} \, J |\log J|^{1/2}  
\ee
for some constant $C > 0$ independent of $D$. Finally, choose the function $H(D) = e^{C(1 + D)}$, noting that this is an increasing function of $D$. Then $\widetilde \Gamma \subset \Gamma$, and the proof is complete.

\end{proof}

\begin{prop}
Let $\epsilon \in (0, \frac{1}{d})$. Assume that $r = r(N) \geq b N^{-1/d + \epsilon}$. Then, for any $\sigma \in (0, d \epsilon)$, for all $N \in \mb{N}$,
\be 
\PP \left ( \|{\bf X}^N-{\bf Y}^N\|_{L^\infty(0,t)} +\|{\bf V}^N-{\bf W}^N\|_{L^\infty(0,t)} \geq r \right ) \lesssim_{\e, \sigma, b} e^{-C_{b} N^\sigma} .
\ee

\end{prop}
\begin{proof}
Let $\theta : = \frac{d(d \epsilon   - \sigma )}{2( 1 - d\epsilon)} > 0$. For all $N$ large enough (depending on $b$ and $\e$), $r \leq \frac{1}{e}$: in this case, set $\gamma = \delta = r^\theta$ and note that $\gamma, \delta < 1$. Then define $J$ as in Definition~\ref{def:J} with $\alpha_0 = r \delta$, $\beta_0 = r \gamma$.

By Lemma~\ref{lem:J_integral},
\be
- |\log J(t)|^{1/2} \leq - | \log J (0) |^{1/2} +  \int_0^t 1 + \frac{ \max_i | E^{{\bf X}}(X_i^N(s)) - E^f(Y_i^N(s))|}{2 J(s) | \log J (s) |^{1/2} } \one_{\{ J(s) < r(1 + \gamma + \delta)  \}}  \dd s .
\ee
We apply Proposition~\ref{prop:Field}: there exists
a non-decreasing function $H$ such that, for each $s \in [0,t]$, the event
\be
\Gamma (s) : = \left \{ \max_i | E^{{\bf X}}(X_i^N(s)) - E^f(Y_i^N(s))|  \leq H(D(s)) |\log J(s) |^{1/2} J(s) \right \} 
\ee
satisfies
\be \label{est:prob_Gamma}
\PP(\Gamma(s)^c \cap \{ J(s) < r(1 + \gamma + \delta)  \}) \lesssim (r \delta)^{-d} \exp \left (- C N r^d \frac{\gamma^2}{1 + \gamma} \right ) ,
\ee
where we have used the fact that $J(s) < r(1 + \gamma + \delta)$ implies that $\|{\bf X}^N - {\bf Y}^N \|_{L^\infty(0,s)} < r$ (Lemma~\ref{lem:Jbasics}(iii)).
Moreover, on the event $\Gamma(s)^c$ we have the bound
\be
\max_i | E^{{\bf X}}(X_i^N(s)) - E^f(Y_i^N(s))| \lesssim  r^{-d} .
\ee
Hence
\be
- |\log J(t)|^{1/2}  + | \log J (0) |^{1/2}  
\leq  \int_0^t 1 + H(D(s)) + \frac{C r^{-d}}{J(s) | \log J (s) |^{1/2} }  \one_{\Gamma(s)^c \cap \{ J(s) < r(1 + \gamma + \delta)  \}}  \dd s .
\ee
Since $J(s) | \log J (s) |^{1/2} \geq J(s) \gtrsim r^{1+\theta}$, we deduce that (redefining the function $H$ as needed),
\be
- |\log J(t)|^{1/2}  + | \log J (0) |^{1/2} \leq  \int_0^t H(D(s)) + C r^{-d-1-\theta} \one_{\Gamma(s)^c \cap \{ J(s) < r(1 + \gamma + \delta)  \}} \dd s .
\ee
Since $H$ and $D$ are non-decreasing functions of their respective arguments,
\be
- |\log J(t)|^{1/2}  + | \log J (0) |^{1/2} \leq  t H(D(t)) + C r^{-d-1-\theta} \int_0^t \one_{\Gamma(s)^c \cap \{ J(s) < r(1 + \gamma + \delta)  \}} \dd s .
\ee

From this, we deduce a one-sided bound in probability. Subtracting $t H(D(t))$ from both sides and taking the expectation of the positive part gives
\be
\EE \left ( - |\log J(t)|^{1/2}  + | \log J (0) |^{1/2} - t H(D(t)) \right )_+ \lesssim r^{-d-1-\theta} \int_0^t   \PP( \Gamma(s)^c \cap \{ J(s) < r(1 + \gamma + \delta)  \})  \dd s .
\ee
Then, by \eqref{est:prob_Gamma}, we obtain the estimate
\be
\EE \left ( - |\log J(t)|^{1/2}  + | \log J (0) |^{1/2} - t H(D(t)) \right )_+ \lesssim t r^{- d -(d +1)(1 + \theta)} \exp \left (- C N r^{d + 2 \theta} \right ) .
 \ee
Hence, by Markov's inequality, 
 \be
 \PP \left [ \left ( - |\log J(t)|^{1/2}  + | \log J (0) |^{1/2} - t H(D(t)) \right )_+ > t \right ] \lesssim  r^{- d -(d +1)(1 + \theta)} \exp \left (- C N r^{d + 2 \theta} \right ) .
 \ee
 
 To complete the proof, we show that if $ \left ( - |\log J(t)|^{1/2}  + | \log J (0) |^{1/2} - t H(D(t)) \right )_+ \leq t$, then $\|{\bf X}^N-{\bf Y}^N\|_{L^\infty(0,t)} +\|{\bf V}^N-{\bf W}^N\|_{L^\infty(0,t)} < r$ for all sufficiently small $r$.
 In this case, we have
 \be
 - |\log J(t)|^{1/2}  + | \log J (0) |^{1/2} \leq \left ( 1 + H(D(t)) \right ) t,
 \ee
 which implies that
  \be \label{est:Jfinal_J0}
 J(t)  \leq \exp \left [ -  \left (   | \log J (0) |^{1/2} - (1 + H (D(t)) ) t \right )^2_+ \right ] .
 \ee 
We need to show that $J(t)$ is strictly smaller than the truncation level $r + \alpha_0 + \beta_0$, in order to ensure that it controls the distance between the particle systems (recall Definition~\ref{def:J}).
 
 To do this, we estimate $J(0)$: since we chose $\alpha_0 = r \delta = r^{1 + \theta}$ and $\beta_0 = r \gamma = r^{1 + \theta}$ above, and the initial data of the two particle systems coincide, we have
 \be \label{recall_def:J0}
J(0) = (r + 2 r^{1 + \theta} ) \wedge \left [ \left (1 +  \sqrt{|\log J(0) |} \right ) r^{1 + \theta} \right ] .
\ee 
Consequently $J(0) > r^{1+\theta}$ and hence $ | \log J (0) |^{1/2} \leq  |\log r^{1+\theta}|^{1/2}$. By substituting this bound back into \eqref{recall_def:J0}, we obtain that  $J(0) \leq r^{1+\theta} ( 1+  |\log r^{1+\theta}|^{1/2} )$. 
Then
\be \label{est:logJ0final}
| \log J(0)|^{1/2} \geq \left ( | \log r^{1+\theta}| - \log ( 1+  |\log r^{1+\theta}|^{1/2} ) \right )^{1/2}\geq  | \log r^{1+\theta}|^{1/2} -  |\log ( 1+  |\log r^{1+\theta}|^{1/2} )|^{1/2} ,
\ee
where we have used the inequality ${(y - x)^{1/2}} \geq y^{1/2} - x^{1/2}$ for $y \geq x \geq 0$, which may be applied here since $\log (1+z^{1/2}) \leq z^{1/2} \leq z$ for all $z \geq 1$.

From \eqref{est:Jfinal_J0} and \eqref{est:logJ0final}, we then have
 \be
r^{-1} J(t)  \leq  \exp \left [|\log r|  -  \left (  | \log r^{1+\theta}|^{1/2} -  |\log ( 1+  |\log r^{1+\theta}|^{1/2} )|^{1/2} - (1 + H(D(t)) ) t \right )^2_+ \right ] .
 \ee
 It follows that $r^{-1} J < 1$ for $r$ such that 
 \be (\sqrt{1+\theta} - 1)|\log r |^{1/2} - |\log ( 1+  |\log r^{1+\theta}|^{1/2} )|^{1/2} - \left (1 + H(D(t)) \right ) t > 0,
 \ee 
 which is the case for all $r < r_\ast$, for some $r_\ast$ depending on $t$, $\theta$ and $D(t)$.
 
 Then, for all $r<r_\ast$, we have $J(t) < r(1 + \gamma + \delta)$, and hence,
 by Lemma~\ref{lem:Jbasics}, $J$ controls the distance between the particle systems:
\be
\|{\bf X}^N-{\bf Y}^N\|_{L^\infty(0,t)} +\|{\bf V}^N-{\bf W}^N\|_{L^\infty(0,t)} \leq J(t) - r(\gamma + \delta) < r .
\ee
Therefore, for $r < r_\ast$,
 \begin{align}
\PP \left ( \|{\bf X}^N-{\bf Y}^N\|_{L^\infty(0,t)} +\|{\bf V}^N-{\bf W}^N\|_{L^\infty(0,t)} \geq r\right )
& \lesssim r^{- d -(d +1)(1 + \theta)} \exp \left (- C N r^{d + 2 \theta} \right ) .
 \end{align}

Now we use the assumption that $r \geq b N^{- \frac{1}{d} + \epsilon}$, which implies that
\be
r^{- d -(d +1)(1 + \theta)} \exp \left (- C N r^{d + 2 \theta} \right ) \lesssim_b \exp \left ( C'_{\theta} |\log N| - C_b N^{d \epsilon   - 2 \theta ( \frac{1}{d} - \epsilon) } \right ) .
\ee
Thus for sufficiently large $N$ (depending on $t$, $D(t)$, $b$, $\epsilon$ and $\sigma$),
\be
\PP \left ( \|{\bf X}^N-{\bf Y}^N\|_{L^\infty(0,t)} +\|{\bf V}^N-{\bf W}^N\|_{L^\infty(0,t)} \geq r \right ) \lesssim_b \exp \left (  - C_b N^{d \epsilon   - 2 \theta ( \frac{1}{d} - \epsilon) } \right ) \lesssim_b \exp \left (  - C_b N^{\sigma} \right ) .
\ee 

For small $N$, the probability is bounded by one: hence, by possibly adjusting the constant in front of the exponential we conclude that 
\be
\PP \left ( \|{\bf X}^N-{\bf Y}^N\|_{L^\infty(0,t)} +\|{\bf V}^N-{\bf W}^N\|_{L^\infty(0,t)} \geq r \right ) \lesssim_{t, D(t), \e, b, \sigma} \exp \left (  - C_b N^{\sigma} \right ) \quad \text{for all} \; \; N \geq 1. 
\ee 
\end{proof}

\section{Convergence to the Ionic Vlasov-Poisson System} \label{sec:VPMEConv}

We conclude this article with the proof of Theorem~\ref{thm:VPMEconv}, which allows us to deduce that the particle system converges not only towards the mean-field dynamics with regularised interaction \eqref{eq:vpme-app}, but to a solution of the original mean-field system \eqref{eq:vpme}.
The result is a consequence of the new estimates in Proposition~\ref{prop:PotentialStability}, paired with the methods of \cite[Proposition 9.1]{LazaroviciPickl}] for the electron Vlasov-Poisson system.

We recall the notation $(Y^r, W^r)$ for the characteristic system associated to the regularised ion Vlasov-Poisson system \eqref{eq:vpme-app}: for all $r \in (0, \frac{1}{4})$ and $(y,w) \in \TT^d \times \R^d$,
\be \label{ODE:FlowReg}
\begin{cases}
\frac{\dd }{\dd t} Y^r(t ; y,w) = W^r(t ; y,w) \\
\frac{\dd }{\dd t} W^r(t ; y,w) = \nabla \Phi_r[\rho_{f_r}] (t, Y^r(t ; y,w)) .
\end{cases}
\ee
Theorem~\ref{thm:VPMEconv} follows from the following proposition.

\begin{prop}
Suppose that Assumption~\ref{hyp:f} holds.
\begin{enumerate}[(i)]
\item There exists a continuous map $(\overline Y, \overline W) : [0, T) \times \TT^d \times \R^d \to \TT^d \times \R^d$ such that $(Y^r, W^r)$ converges to $(\overline Y, \overline W)$ uniformly on $[0,t] \times \TT^d \times \R^d$ for any $t \in [0, T)$, with
\be
\label{est:FlowConvRate_Statement}
 \|Y^r- \overline Y\|_{L^\infty([0,t] \times \TT^d \times \R^d)} + \| W^r - \overline W \|_{L^\infty([0,t] \times \TT^d \times \R^d)} \leq r e^{C(t) |\log r|^{1/2}} .
\qquad t \in [0, T)
\ee
for some function $C \in L^\infty_\loc[0,T)$, which depends on the function $D$ from Assumption~\ref{hyp:f}, the dimension $d$ and the choice of mollifier $\chi$.

\item Let $f \in C \left ( [0,T) ; \mc{P}(\TT^d \times \R^d) \right )$ be the pushforward of $f_0$ along $(\overline Y, \overline W)$. That is, $f$ is 
be a weakly continuous path in the space of probability measures defined by the equality
\be \label{def:f_pushforward}
\int_{\TT^d \times \R^d} \phi(y,w) f (t, y,w) \dd y \dd w = \int_{\TT^d \times \R^d} \phi (\overline Y(t; y,w), \overline W(t; y,w) f_0 (y,w) \dd y \dd w 
\ee
for all $t \in [0,T)$ and $\phi \in C_b(\TT^d \times \R^d)$. Then $f$ is the unique weak solution of the ionic Vlasov-Poisson system \eqref{eq:vpme} with initial datum $f_0$ and satisfying $\rho_f \in L^\infty_\loc( [0, T) ; L^\infty(\TT^d))$.

Moreover $(\overline Y, \overline W)$ satisfies
\be \label{ODE:FlowBar}
\begin{cases}
\frac{\dd }{\dd t} \overline Y (t ; y,w) = \overline W(t ; y,w) \\
\frac{\dd }{\dd t} \overline W (t ; y,w) = \nabla \Phi [\rho_{f}] (t, \overline Y(t ; y,w)) .
\end{cases}
\ee
\end{enumerate}
\end{prop}
\begin{proof}
Let $r , q > 0$ and consider the flows $(Y^r, W^r)$ and $(Y^q, W^q)$, each of which is a continuous map from $[0, T) \times \TT^d \times \R^d$ to $\TT^d \times \R^d$.

For any given point $(y,w ) \in \TT^d \times \R^d$, by
taking the difference of the ODEs \eqref{ODE:FlowReg} for $(Y^r, W^r)$ and $(Y^q, W^q)$ we can obtain, for any $t \in [0,T)$, the estimates
\begin{align}
& | Y^r(t ; y, w  ) - Y^q(t ; y,w ) | \leq \int_0^t | W^r(s ; y, w  ) - W^q(s ; y,w ) | \dd s \\
&| W^r(t ; y, w  ) - W^q(t ; y,w ) | \leq \int_0^t \left \| \nabla \Phi_r[\rho_{f_r}] (s, \cdot ) - \nabla  \Phi_q[\rho_{f_q}] (s, \cdot ) \right \|_{L^\infty(\TT^d)} \dd s .
\end{align}
Taking the essential supremum over $(y,w ) \in \TT^d \times \R^d$, we obtain
\begin{align} \label{est:FlowDiff_sup1}
& \| Y^r(t ; \cdot  ) - Y^q(t ; \cdot ) \|_{L^\infty(\TT^d \times \R^d)} \leq \int_0^t \| W^r(s ; \cdot  ) - W^q(s ; \cdot ) \|_{L^\infty(\TT^d \times \R^d)} \dd s \\
&\| W^r(t ; \cdot  ) - W^q(t ; \cdot ) \|_{L^\infty(\TT^d \times \R^d)} \leq \int_0^t \left \| \nabla \Phi_r[\rho_{f_r}] (s, \cdot ) - \nabla  \Phi_q[\rho_{f_q}] (s, \cdot ) \right \|_{L^\infty(\TT^d)} \dd s .
\end{align}

To estimate $\left \| \nabla \Phi_r[\rho_{f_r}] - \nabla  \Phi_q[\rho_{f_q}]  \right \|_{L^\infty(\TT^d)}$, we seek a bound on $\| K_r \ast \rho_{f_r} - K_q \ast \rho_{f_q} \|_{L^\infty(\TT^d)}$.
The following estimate is shown in the proof of \cite[Proposition 9.1]{LazaroviciPickl} for $x \in \R^d$, and can be adapted straightforwardly to the case $x \in \TT^d$: if $\tilde f_r$ denotes the pushforward of $f_0$ along the flow $(Y^r, W^r)$, and $\tilde f_q$ denotes the pushforward of $f_0$ along the flow $(Y^q, W^q)$, then
\be \label{est:KDiff_Uniform}
\| K_r \ast \rho_{\tilde f_r(t)} - K_r \ast \rho_{\tilde f_q(t)} \|_{L^\infty(\TT^d)} \lesssim_{d, \chi} |\log r| \| Y^r(t; \cdot) - Y^q(t; \cdot) \|_{L^\infty(\TT^d \times \RR^d)}  ( \| \rho_{\tilde f_r(t)} \|_{L^\infty (\TT^d)} + \| \rho_{\tilde f_q(t)} \|_{L^\infty (\TT^d)}) .
\ee
In this case $\tilde f_r = f_r$ and $\tilde f_q = f_q$, and hence by Assumption~\ref{hyp:f} we have
\be
\| K_r \ast \rho_{\tilde f_r(t)} - K_r \ast \rho_{\tilde f_q(t)} \|_{L^\infty(\TT^d)} \lesssim_{d, \chi} D(t) |\log r| \| Y^r(t; \cdot) - Y^q(t; \cdot) \|_{L^\infty(\TT^d \times \RR^d)}  .
\ee
Since $\| K_r - K_q \|_{L^1(\TT^d)} \leq \| K_r - K \|_{L^1(\TT^d)} + \| K - K_q \|_{L^1(\TT^d)}\lesssim_{d, \chi} r + q$, it follows that
\be
\| (K_r - K_q) \ast \rho_{\tilde f_r(t)} \|_{L^\infty(\TT^d)} \lesssim_{d, \chi}  D(t) (r + q) .
\ee
Hence
\be
\| K_r \ast \rho_{f_r(t)} - K_q \ast \rho_{f_q(t)} \|_{L^\infty(\TT^d)} \lesssim_{d, \chi}  D(t) \left ( |\log r| \| Y^r(t; \cdot) - Y^q(t; \cdot) \|_{L^\infty(\TT^d \times \RR^d)} + r + q \right ) .
\ee
By Proposition~\ref{prop:PotentialStability}, we deduce that
\be
\left \| \nabla \Phi_r[\rho_{f_r(t)}] - \nabla  \Phi_q[\rho_{f_q(t)}]  \right \|_{L^\infty(\TT^d)}  \leq e^{C_{d, \chi} D(t)}  \left ( |\log r| \| Y^r(t; \cdot) - Y^q(t; \cdot) \|_{L^\infty(\TT^d \times \RR^d)} + r + q \right ) .
\ee

Substitute this bound into \eqref{est:FlowDiff_sup1} and take supremum over times less than or equal to $t$ to obtain
\begin{align}
& \| Y^r - Y^q \|_{L^\infty([0,t] \times \TT^d \times \R^d)} \leq \int_0^t \| W^r - W^q \|_{L^\infty([0,s] \times ]\TT^d \times \R^d)} \dd s \\
&\| W^r - W^q \|_{L^\infty([0,t] \times \TT^d \times \R^d)} \leq \int_0^t C(s) \left ( |\log r| \| Y^r - Y^q \|_{L^\infty([0,s] \times \TT^d \times \RR^d)} + r + q \right )  \dd s ,
\end{align}
for some locally bounded function $C : [0,T) \to [0, +\infty)$ depending on $D, d, \chi$.
Hence, if we define the quantity
\be
J(t) : = |\log r |^{1/2}  \|Y^r- Y^q\|_{L^\infty([0,t] \times \TT^d \times \R^d)} +  \| W^r - W^q \|_{L^\infty([0,t] \times \TT^d \times \R^d)} ,
\ee
then
\be
J(t) \leq  \int_0^t C (s) \left ( |\log r|^{1/2} J(s) + r + q \right ) \dd s .
\ee
Therefore, by the integral form of Gr\"onwall's lemma,
\be \label{est:FlowDiff_Cauchy}
J(t) \leq e^{C(t) |\log r|^{1/2}} (r+q) ,
\ee
where $C : [0,T) \to [0, +\infty)$ is a (possibly different) function depending on $d, \chi$ and $D$.

It follows from \eqref{est:FlowDiff_Cauchy} that, for any sequence $r_n > 0$ tending to zero as $n$ tends to infinity, $( Y_{r_n}, W_{r_n})$ is a Cauchy sequence in $C( [0,t] \times \TT^d \times \R^d)$ for any $t \in [0, T)$. A limit $(\overline Y, \overline W)$ therefore exists.
By  passing to the limit in \eqref{est:FlowDiff_Cauchy}, we obtain that for any $r \in (0, \frac{1}{4})$
\be \label{est:FlowConvRate}
 \|Y^r- \overline Y\|_{L^\infty([0,t] \times \TT^d \times \R^d)} + \| W^r - \overline W \|_{L^\infty([0,t] \times \TT^d \times \R^d)} \leq r e^{C(t) |\log r|^{1/2}} .
\ee
This also implies that $(\overline Y, \overline W)$ is the unique limit point. By using this argument for arbitrary $t \in [0, T)$, we may therefore extend to a well-defined limit function $(\overline Y, \overline W) \in C( [0, T) \times \TT^d \times \R^d)$.

Let $f$ by defined as in \eqref{def:f_pushforward}. It remains to show that $f$ is solution of the ionic Vlasov-Poisson system \eqref{eq:vpme}.
We will prove that the electric fields along the flow converge as $r$ tends to zero: namely, that
\be \label{eq:FieldConv}
\lim_{r \to 0} \| \nabla \Phi_r[\rho_{f_r}] (Y^r ) - \nabla \Phi[\rho_f] (\overline Y) \|_{L^\infty ([0,t] \times \TT^d \times \R^d) } = 0 .
\ee
Given \eqref{eq:FieldConv}, we may pass to the limit in the ODE \eqref{ODE:FlowReg} to find that the limiting flow satisfies \eqref{ODE:FlowBar}.
This is sufficient to deduce that $f$ is a weak solution of \eqref{eq:vpme}, using the superposition principle for the representation of solutions of transport equations in terms of characteristic trajectories \cite{DiPL, Ambrosio, AGS}. 
The uniqueness of solutions of \eqref{eq:vpme} with $\rho_f$ bounded in $L^\infty(\TT^d)$ was proved in \cite{GPIWP} for $d=2, \, 3$, and the proof can be adapted to $d > 3$ using the $W_2$ stability estimates in \cite{GPIQN}.

To establish \eqref{eq:FieldConv}, first observe that, for any $(y,w) \in \TT^d \times \R^d$
\begin{multline}
| \nabla \Phi_r[\rho_{f_r}] (Y^r(t; y, w ) - \nabla \Phi[\rho_f] (\overline Y(t; y,w) ) | \leq | \nabla \Phi_r[\rho_{f_r}] (Y^r(t; y, w ) - \nabla \Phi_r[\rho_{f_r}] (\overline Y(t; y,w) ) | \\
+ | \nabla \Phi_r[\rho_{f_r}] (\overline Y(t; y, w ) - \nabla \Phi[\rho_f] (\overline Y(t; y,w) ) | .
\end{multline}
By Proposition~\ref{prop:potential}, 
\begin{multline}
| \nabla \Phi_r[\rho_{f_r}] (Y^r(t; y, w ) - \nabla \Phi[\rho_f] (\overline Y(t; y,w) ) | \leq D(t) |Y^r(t; y, w ) - \overline Y(t; y,w)  | \log |Y^r(t; y, w ) - \overline Y(t; y,w)  |^{-1} \\
 + | [ \nabla \Phi_r[\rho_{f_r}] - \nabla \Phi[\rho_f] ] (\overline Y(t; y,w) ) | .
\end{multline}
Hence
\begin{multline}
\| \nabla \Phi_r[\rho_{f_r}] (Y^r ) - \nabla \Phi[\rho_f] (\overline Y) \|_{L^\infty ([0,t] \times \TT^d \times \R^d) } \leq D(t) \|Y^r- \overline Y\|_{L^\infty([0,t] \times \TT^d \times \R^d)} \log \|Y^r- \overline Y\|_{L^\infty([0,t] \times \TT^d \times \R^d)}^{-1} \\
+ \| \nabla \Phi_r[\rho_{f_r}] - \nabla \Phi[\rho_f] \|_{L^\infty ([0,t] \times \TT^d \times \R^d) } .
\end{multline}

We note that $\| \rho_f \|_{L^\infty(\TT^d)} \leq D(t)$: indeed, $\rho_{f_r(t)}$ converges to $\rho_{f(t)}$ in the sense of weak convergence of measures on $\TT^d$, and thus
\be
\int_{\TT^d} \phi(x)  \rho_f(t,x) \dd x \leq \limsup_{r \to 0} \| \rho_{f_r} \|_{L^\infty (\TT^d)} \| \phi \|_{L^1(\TT^d)} \leq D(t) \| \phi \|_{L^1(\TT^d)} \qquad \forall \phi \in C(\TT^d) 
\ee
and we conclude by density of $C(\TT^d)$ in $L^1(\TT^d)$.

Now consider 
\be
\| K_r \ast \rho_{f_r} - K \ast \rho_f \|_{L^\infty(\TT^d)} \leq \| K_r \ast \rho_{f_r} - K_r \ast \rho_f \|_{L^\infty(\TT^d)} + \| (K_r - K ) \ast \rho_f \|_{L^\infty(\TT^d)}.
\ee
We have $\| (K_r - K ) \ast \rho_f \|_{L^\infty(\TT^d)} \lesssim_{d, \chi} D(t) r$.
By \eqref{est:KDiff_Uniform}, 
\be
\| K_r \ast \rho_{f_r} - K_r \ast \rho_f \|_{L^\infty(\TT^d)} \lesssim_{d, \chi} D(t) |\log r| \|Y^r- \overline Y\|_{L^\infty([0,t] \times \TT^d \times \R^d)} .
\ee
We conclude by \eqref{est:FlowConvRate} that
\be
\| K_r \ast \rho_{f_r} - K \ast \rho_f \|_{L^\infty(\TT^d)}\leq r e^{ C(t) |\log r|^{1/2}} 
\ee
for some function $C \in L^\infty_\loc [0,T)$.
Then, by Proposition~\ref{prop:PotentialStability},
\be
\| \nabla \Phi_r[\rho_{f_r}] - \nabla \Phi[\rho_f] \|_{L^\infty ([0,t] \times \TT^d \times \R^d) } \leq r e^{C(t) |\log r|^{1/2}} 
\ee
(for a possibly different function $C \in L^\infty_\loc [0,T)$).

We conclude that
\be
\| \nabla \Phi_r[\rho_{f_r}] (Y^r ) - \nabla \Phi[\rho_f] (\overline Y) \|_{L^\infty ([0,t] \times \TT^d \times \R^d) } \leq  C(t) \left (1 + |\log r| \right ) r e^{C(t) |\log r|^{1/2}} ,
\ee
and the proof is complete.

\end{proof}

\begin{paragraph}{Acknowledgements.}
The author gratefully acknowledges the support of the Additional Funding Programme for Mathematical Sciences, delivered by EPSRC (EP/V521917/1) and the Heilbronn Institute for Mathematical Research. Aspects of this work were undertaken during the programme `Frontiers in kinetic theory: connecting microscopic to macroscopic scales - KineCon 2022' (supported by EPSRC grant EP/R014604/1) at the Isaac Newton Institute for Mathematical Sciences, Cambridge; as well as during a visit to the Forschungsinstitut für Mathematik at ETH Zürich, both of which the author thanks for support and hospitality.

\end{paragraph}

\bibliographystyle{abbrv}
\bibliography{MFLbib}

\end{document}